\documentclass[submission,copyright,creativecommons,bookmarks,bookmarksopen,bookmarksdepth=2]{eptcs}
\providecommand{\event}{ACT 2023} % Name of the event you are submitting to
\usepackage{amsmath,amssymb,amsthm}
\usepackage{mathtools}
\usepackage{amsfonts}
\usepackage{iftex}
\usepackage{hyperref}
\usepackage{booktabs} % for improved spacing around horizontal lines, incompatible with vertical lines, be careful if you want to to combine it with color

\usepackage[dvipsnames]{xcolor}

\ifpdf
  \usepackage{underscore}         % Only needed if you use pdflatex.
  \usepackage[T1]{fontenc}        % Recommended with pdflatex
\else
  \usepackage{breakurl}           % Not needed if you use pdflatex only.
\fi
\usepackage{cleveref}

\hypersetup{
  colorlinks   = true,    % Colours links instead of ugly boxes
  urlcolor     = black,    % Colour for external hyperlinks
  linkcolor    = black,    % Colour of internal links
  citecolor    = black      % Colour of citations
}

\DeclareSymbolFont{newfont}{OML}{cmm}{m}{it}% Computer Modern math font
\DeclareMathSymbol{\Epsilon}{3}{newfont}{15}% Symbol 15
\title{Unifilar Machines and the Adjoint Structure of Bayesian Filtering}
\author{Nathaniel Virgo
\institute{Earth-Life Science Institute (ELSI), Tokyo Instutute of Technology}
\email{nathanielvirgo+science@gmail.com}
}
\def\titlerunning{Adjoint Structure of Bayesian Filtering}
\def\authorrunning{Nathaniel Virgo}

\usepackage{globalvals}
\newcommand{\defFootnote}{\defVal}
\newcommand{\insFootnote}[1]{\footnote{\useVal{#1}}}
\DeclareFontFamily{U}{stix2bb}{\skewchar\font127 }
\DeclareFontShape{U}{stix2bb}{m}{n} {<-> stix2-mathbb}{}
\DeclareMathAlphabet{\mathbb}{U}{stix2bb}{m}{n}

\usepackage{tikz-cd}

\usepackage{cleveref}

\newtheorem{theorem}{Theorem}[section]
\newtheorem{proposition}[theorem]{Proposition}

\theoremstyle{definition}
\newtheorem{definition}[theorem]{Definition}
\newtheorem{example}[theorem]{Example}

\usepackage{stringdiagsexpt_v03}

\usepackage{outlines}

\newenvironment{aligneq}{
    \begin{equation}
    \begin{aligned}
}{
    \end{aligned}
    \end{equation}
    \ignorespacesafterend
}

\usepackage[mathscr]{euscript}
\newcommand{\cat}[1]{\ensuremath{\mathscr{#1}}}

\newcommand{\onecat}[1]{\ensuremath{\mathbf{#1}}}

\newcommand{\oc}{\onecat}

\newcommand{\opn}{\operatorname}

\usepackage{stmaryrd}
\newcommand{\comp}{\mathbin{\fatsemi}}

\newcommand{\id}{\mathrm{id}}

\newcommand{\mor}{\xrightarrow}

\newcommand{\blank}{{-}}
\newcommand{\bblank}{{=}}

\newcommand{\R}{\ensuremath{\mathbb{R}}}

\renewcommand{\det}{\text{det}}
\newcommand{\samp}{\opn{samp}}

\newcommand{\cuttext}[1]{}

\usepackage{bm}
\newcommand{\squ}{\text{\scalebox{0.6}{$\,\bm\square\,$}}}

\newcommand{\blob}{\bullet}

\begin{document}
\maketitle

\begin{abstract}
We elucidate the mathematical structure of Bayesian filtering, and Bayesian inference more broadly, by applying recent work on category theoretical probability, specifically the concept of a strongly representable Markov category.
We show that filtering, along with related concepts such as conjugate priors, arise from an adjunction: the process of taking a hidden Markov process is right adjoint to a forgetful functor.
This has an interesting consequence.
In practice, filtering is usually implemented using parametrised families of distributions.
The Kalman filter is a particularly important example, which uses Gaussians.
Rather than calculating a new posterior each time, the implementation only needs to udpate the parameters.
This structure arises naturally from our adjunction; the correctness of such a model is witnessed by a map from the model into the system being modelled.
Conjugate priors arise from this construction as a special case.

In showing this we define a notion of \emph{unifilar machine}, which has its origins in the literature on $\Epsilon$-machines.
Unifilar machines are useful as models of the `observable behaviour' of stochastic systems; we show additionally that in the Kleisli category of the distribution monad there is a terminal unifilar machine, and its elements are controlled stochastic processes, mapping sequences of the input alphabet probabilistically to sequences of the output alphabet.
\end{abstract}

\section{Introduction}

This paper is concerned with the mathematical structure of \emph{Bayesian filtering}, which is a common problem in applications of Bayesian inference.
The idea is that there is some system with known dynamics (which in general are stochastic) but an unknown hidden state.
The goal is to keep track of a Bayesian prior over the states of the system, updating it to a posterior whenever a new observation is made.
This is useful if we want to be able to control the hidden state, as in solving a partially observable Markov decision process (POMDP), for example.

To reveal the underlying mathematical structure we make use of recent results in synthetic probability, which allows us to write proofs at the category theoretic level without using measure theory directly.
We work in the framework of Markov categories \cite{fritz2020-markov}, and in particular we make use of the concept of \emph{strongly representable Markov category} as defined in \cite{fritz2020-representable}.
Strongly representable Markov categories include \oc{BorelStoch} (whose objects are standard Borel spaces and whose morphisms are Markov kernels) and the Kleisli category of the (real-valued) distribution monad, which we refer to as \oc{Dist}.
Therefore most of our results apply to both measure-theoretic probability and finitely supported probability.

We model a system with a hidden state as a certain kind of stochastic Moore machine (essentially a hidden Markov model); we refer to this informally as a \emph{dynamical model} of the system.
There is then a functor $B$ that takes such a dynamical model and maps it to an \emph{epistemic model}.
This lives in a different category of machines that we call \emph{unifilar machines}, whose outputs are stochastic but whose state updates are deterministic.
Its state space consists of probability distributions over the hidden states of the system, and its dynamics are given by Bayesian updating.

Our main technical result is \cref{adjunction-result}, which states that this functor is right adjoint to a forgetful functor in the opposite direction.
This has an interesting consequence.
The functor $B$ maps a dynamical model $\kappa$ to what could be called its {universal epistemic model}, $B(\kappa)$.
By this we mean that if we consider another unifilar machine $\alpha$ equipped with a morphism $\alpha\to B(\kappa)$, we can also consider this an epistemic model of $\kappa$, in a sense that we describe next.

In applications, one doesn't necessarily want to keep track of the Bayesian distribution directly.
Instead, one uses a parametrised family of distributions, chosen such that the update step only needs to update the parameters to produce the posterior distribution.
For this to work, the Bayesian posterior must always be in the same family of distributions as the prior.
An example with enormous practical importance is the Kalman filter, \cref{kalman-filter-example}.
Here the prior is a multivariate Gaussian and the posterior is always also a multivariate Gaussian.
The filter's state space consists of the parameters of such a Gaussian, and the update step simply maps them to their new values.
In our framework this kind of structure arises from considering a morphism $\alpha\to B(\kappa)$.
The state space of $B(\kappa)$ consists of probability distributions, and the state space of $\alpha$ consists of values that parametrise them in a consistent way.

This idea is closely related to the notion of conjugate prior, which was previously studied in a category-theoretic context in \cite{jacobs2020-conjugate}.
The definition in that paper is essentially our \cref{jacobs-equation-per-se}, which arises from our framework in a very natural way.
The connection between Bayesian filtering and Bayesian inference is explored in \cref{bayesian-inference-section}, where we also briefly touch on connections to recent work on de Finetti's theorem within a category-theoretic context \cite{jacobs2020-de-finetti,fritz2021-de-finetti}.

A secondary contribution of our paper is an exploration of the possible generalisations of Moore machines to the stochastic case.
Our result involves two different generalisations of Moore machine, which we term \emph{comb machine} (\cref{comb-machine-definition}) and \emph{unifilar machine} (\cref{unifilar-comb-machine-definition}).
Unifilar machines in particular are of independent interest.
They are based on an idea from the literature on $\Epsilon$-machines \cite{barnett2015-epsilon}.
%\insFootnote{epsilon-machines}.
%
They are defined such that their output map is stochastic but their update map is (almost surely) deterministic given their input and their output.
This means that their states map more directly to `behaviours' than the states of a more general stochastic machine.
Indeed, we show in \cref{transducer-section} that in \oc{Dist} the category of unifilar machines has a terminal object,
%(for every choice of input and output space),
which consists of the collection of `controlled stochastic processes,' also known as `stochastic streams' \cite{diLavore2022-coinductive}. 
In general, if a category of unifilar machines has a terminal object then it can be seen as an ``object of behaviours'' of stochastic systems.%, including hidden Markov models as well as other unifilar machines.

Our Bayesian filtering machines have a strong resemblance to Bayesian lenses as presented in \cite{smithe2020-bayesian,braithwaite2023-compositional} (see also \cite{kamiya2021-framework}, which develops the idea in a way that makes all of the relevant maps measurable).
Our work is closely related, since the machine $B(\kappa)$ described above has the same data as a Bayesian lens.
However, general unifilar machines do not seem to compose like lenses, and working out the precise relationship is a task for future work.

Our work also seems related to the notion of \emph{determinisation} from automata theory.
Here one takes a nondeterministic automaton (one that may have more than one transition from a given state, but with no notion of probability assigned to them) and turns it into a deterministic automaton whose state space is the power set of the original automaton.
Determinisation has been studied and generalised in a coalgebraic context \cite{silva2013-determinization, adamek2012-coalgebraic};
understanding the exact relation to our work is also a task for future work.

Bayesian filtering and its connection to conjugate priors was previously considered in a Markov category context by the author and colleagues in \cite{virgo2021-interpreting}.
A related approach to Bayesian filtering in Markov categories was presented at the conference \cite{fritz2023-hidden-markov}; the present work was developed independently.
The novel contribution of the present paper is to reveal more of the abstract categorical structure underlying the idea, including the definitions of comb machine, unifilar machine and the adjoint structure involving the functor $B$, as well as the brief discussion of terminal unifilar machines in \cref{transducer-section}.

\subsection{Background on Representable Markov categories}

\interfootnotelinepenalty=10000 % prevent breaking of footnotes across pages

We will use the machinery of representable Markov categories and in particular, strongly representable Markov categories, both defined in \cite{fritz2020-representable}.\insFootnote{definition-only-in-preprint}
For general background on Markov categories we refer to \cite{fritz2020-markov}.
Definitions in this section are from the literature, except for \cref{generalised-almost-surely}, which is a slight generalisation of the usual category-theoretic definition of `almost surely.'

\defFootnote{definition-only-in-preprint}{
The definitions of strongly representable Markov categories and 
`deterministic given $X$' appear in version 2 of the arXiv preprint 
\cite{fritz2020-representable} but not in the published version of the 
same paper (\cite{fritz2023-representable-published}) or version 3 of the 
preprint. %
The author understands that the removal was for reasons of narrative structure rather than any technical defect.
}

Recall from \cite{fritz2020-representable} that given an object $X$ in a Markov category $\cat{C}$, a \emph{distribution object} is an object $PX$ equipped with a map $\opn{samp}_X\colon PX\to X$ such that for every morphism $f\colon A\to X$ there is a unique deterministic morphism $f^\squ \colon A\to PX$ such that $f^\squ \comp \opn{samp_X} = f$.
A Markov category is then called \emph{representable} if every object has a distribution object.
Representable Markov categories often arise as the Kleisli categories of monads obeying conditions spelt out in \cite{fritz2020-representable}.

The two examples we will use are \oc{BorelStoch} (the Kleisli category of the Giry monad, restricted to standard Borel spaces) and the Kleisli category of the (real-valued) distribution monad, which we will call \oc{Dist}.
These are both shown to be representable in \cite{fritz2020-representable}.

We recall also the following results about representable Markov categories:
When every object has a distribution object, $P$ extends to a functor $P\colon \cat{C} \to \cat{C}_\det$.
Restricting the domain of this functor we obtain a functor $P_\det\colon \cat{C}_\det \to \cat{C}_\det$, which will also be written as $P$, except when we wish to explicitly disambiguate. 
The functor $P_\det$ can be made into a monad on $\cat{C}_\text{det}$, and the Kleisli category of this monad is $\cat{C}$.
The unit has components $\delta_X = \id_X^\squ\colon X\to PX$, and the multiplication has components $\mu_X = P(\opn{samp}_X)\colon PPX\to PX$.

This monad arises from an adjunction: the functor $P$ is right adjoint to the inclusion functor $\cat{C}_\det\hookrightarrow \cat{C}$.
Its unit has components $\delta_X$ and its counit has components given by the sampling map $\opn{samp}_X\colon PX\to P$.

\newcommand{\sdstoch}{\sdmorphism}
\newcommand{\sddet}{\sdblack}

In string diagrams we will draw $\opn{samp}_X$ as a white dot.
Additionally, if a morphism is known to be deterministic (\cite{fritz2020-markov}, definition 2.2) we indicate this with a black bar at its right-hand edge, so we can write
\begin{equation}
    \stringdiagram{
        \sdin{inA}{A} \\
        \sdstoch{f}{f}{1}{1} \\
        \sdout{outX}{X} \\
    }{
        \sdconnect{inA}{f}
        \sdconnect{f}{outX}
    }
    =
    \stringdiagram{
        \sdin{inA}{A} \\
        \sddet{f}{f^\squ}{1}{1} \\[1em]
        \sdwhitedot{samp1} \\
        \sdout{outX}{X} \\
    }{
        \sdconnect{inA}{f}
        \sdconnect[PX]{f}{samp1}
        \sdconnect{samp1}{outX}
    }.
\end{equation}

We will need the definition of a strongly representable Markov category.
For this we first recall some more definitions from \cite{fritz2020-markov} and \cite{fritz2020-representable}.

\begin{definition}
[conditionals; \cite{fritz2020-markov}, definition 11.5]
Let $f\colon A\to X\otimes Y$ be a morphism in a Markov category $\cat{C}$.
We say that a morphism $c\colon X\otimes A\to Y$ is a \emph{conditional} of $f$ if
\begin{equation}
    \label{conditional-condition}
    \stringdiagram{
        &[-1mm]\sdin{inA}{A}&[-1mm] \\
        &\sdstoch{f}{f}{1}{2} & \\[1em]
        \sdout{outY}{Y} && \sdout{outX}{X} \\
    }{
        \sdconnect{inA out1}{f in1}
        \sdconnect{f out1}{outY in1}
        \sdconnect{f out2}{outX in1}
    }
    =
    \stringdiagram{
        \sdin{inA}{A} & \\
        \sdblackdot{cp1} & \\
        & \sdstoch{f}{f}{1}{2}  \\
        & \sdinvisn{i1}{2}  \\
        & \sdinvisn{i2}{2}  \\
        \sdstoch[yshift=1mm]{c}{c}{2}{2} & \\
        \sdout{outY}{Y} & \sdout[1mm]{outX}{X} \\
    }{
        \sdblackdotat{i1 out1}
        \sdblackdotat[cp2]{i2 out2}
        \sdconnect{inA out1}{cp1 in1}
        \sdconnect{cp1}{f in1}
        \sdconnect{cp1 out1}{c in1}
        \sdconnect{f out2}{i2 in2}
        \sdconnect{f out1}{i1 in1}
        \sdconnect{cp2}{c in2}
        \sdconnect{cp2 out1}{outX in1}
        \sdconnect{c out1}{outY in1}
    }.
\end{equation}
We say $\cat{C}$ \emph{has conditionals} if every morphism of the appropriate type has a conditional.
\end{definition}
The intuition is that for every value of the parameter $A$, the morphism $f$ defines a joint distribution between $X$ and $Y$, and \cref{conditional-condition} represents a factorisation of this joint distribution into the marginal distribution over $X$ and a  conditional distribution of $Y$ given $X$.
Recall from \cite{fritz2020-markov} that \oc{Dist} and \oc{BorelStoch} both have conditionals, but \oc{Stoch} does not.
Conditionals are in general not unique when they exist; see \cite{fritz2020-markov}, proposition 11.15 and the discussion surrounding it.

The following is essentially definition 13.1 of \cite{fritz2020-markov} (`almost surely'), but we generalise it slightly.
\begin{definition}
    [generalised almost surely]
    \label{generalised-almost-surely}
    Given a morphism $p\colon A\otimes C\to X$ in a Markov category $\cat{C}$, we say morphisms $f,g\colon X\otimes A\to Y$ are \emph{$p$-generalised-almost-surely equal} or \emph{$p$-g.a.s.\ equal} if
    \begin{equation}
    \stringdiagram{
        \sdin{inA}{A} & \sdin{inB}{B} &  \\
        \sdblackdot{cp1} && \\
        && \sdstoch{f}{p}{1}{1}  \\
        &\sdinvis[xshift=-4mm]{i1}& \sdblackdot{cp2} \\[1em]
         \sdstoch[yshift=1mm]{c}{f}{3}{3} && \\
        \sdout[1mm]{outY}{Y} && \sdout{outX}{X} \\
    }{
%        \sdconnect{inC out1}{f in2}
        \sdconnect{inB out1}{i1 in1}
        \sdconnect{i1 out1}{c in2}
        \sdconnect{inA out1}{cp1 in1}
        \sdconnect{cp1}{f in1}
        \sdconnect{cp1 out1}{c in1}
        \sdconnect{f out1}{cp2 in1}
        \sdconnect{cp2}{c in3}
        \sdconnect{cp2 out1}{outX in1}
        \sdconnect{c out2}{outY in1}
    }
    =
    \stringdiagram{
        \sdin{inA}{A} & \sdin{inB}{B} &  \\
        \sdblackdot{cp1} && \\
        && \sdstoch{f}{p}{1}{1}  \\
        &\sdinvis[xshift=-4mm]{i1}& \sdblackdot{cp2} \\[1em]
         \sdstoch[yshift=1mm]{c}{g}{3}{3} && \\
        \sdout[1mm]{outY}{Y} && \sdout{outX}{X} \\
    }{
%        \sdconnect{inC out1}{f in2}
        \sdconnect{inB out1}{i1 in1}
        \sdconnect{i1 out1}{c in2}
        \sdconnect{inA out1}{cp1 in1}
        \sdconnect{cp1}{f in1}
        \sdconnect{cp1 out1}{c in1}
        \sdconnect{f out1}{cp2 in1}
        \sdconnect{cp2}{c in3}
        \sdconnect{cp2 out1}{outX in1}
        \sdconnect{c out2}{outY in1}
    }.
    \end{equation}
\end{definition}
The idea is that in a measure-theoretic context such as \oc{Stoch} or \oc{BorelStoch}, for a values $A$, the values of $f$ and $g$ can differ only on subsets of $X$ that have measure zero according to $p$.
In \oc{Dist} it means that $f(y\mid x,b,a) = g(y\mid x,b,a)$ whenever $p(x\mid a)>0$.

The category-theoretic definition of almost-surely in \cite{fritz2020-markov} allows $p$ but not $f$ or $g$ to depend on a parameter, so this is a slight generalisation.
We avoid calling our version ``almost surely'' to avoid confusion with the usual definition, as used in \cite{fritz2020-markov}, \cite{fritz2020-representable} and other works.

\newcommand{\sdhalfdet}{\sdhalfblack}

\newcommand{\inlineMarginal}[1]{%
$
\stringdiagram{
    \sdinvis{i1} \\
    \sdstoch{f}{#1}{1}{2} \\[-1.5mm]
    \sdinvisn{i2}{2} \\[-1.5mm]
    \sdinvisn{i3}{2} \\
    } {
        \sdblackdotat{i2 in1}
        \sdconnect{i1}{f in1}
        \sdconnect{f out1}{i2 in1}
        \sdconnect{f out2}{i3 in2}
    }
$%
}

\begin{definition}
[deterministic given $X$; \cite{fritz2020-representable}, definition 6.4]
Let $f\colon A\to X\otimes Y$ be a morphism in a Markov category \cat{C}, and let $f$ be such that a conditional $c\colon X\otimes A\to Y$ exists.
The morphism $f$ is said to be \emph{deterministic given~$X$} if the conditional is \inlineMarginal{f}-generalised-almost-surely deterministic, in the sense that
\begin{equation}
    \label{deterministic-given-x-condition}
    \stringdiagram{
        &[-1mm] \sdin{inA}{A} &[-1mm]&[-1mm] \\
        & \sdblackdot{cp1} && \\
        &&& \sdstoch{f}{f}{1}{2}  \\
        &&& \sdinvisn{i1}{2}  \\
        &&& \sdinvisn{i2}{2}  \\
        & \sdstoch[yshift=1mm]{c}{c}{2}{1} & \\
        & \sdblackdot[yshift=1mm]{cp3} & \\
        \sdout{outY1}{Y} && \sdout{outY2}{Y} & \sdout[1mm]{outX}{X} \\
    }{
        \sdblackdotat{i1 out1}
        \sdblackdotat[cp2]{i2 out2}
        \sdconnect{inA out1}{cp1 in1}
        \sdconnect{cp1}{f in1}
        \sdconnect{cp1 out1}{c in1}
        \sdconnect{f out2}{i2 in2}
        \sdconnect{f out1}{i1 in1}
        \sdconnect{cp2}{c in2}
        \sdconnect{cp2 out1}{outX in1}
        \sdconnect{c out1}{cp3 in1}
        \sdconnect{cp3}{outY1 in1}
        \sdconnect{cp3}{outY2 in1}
    }
    =
    \stringdiagram{
        &[-1ex] \sdin{inA}{A} &[-1ex]&[-1mm] \\
        & \sdblackdot{cp1} && \\
        &&& \sdstoch{f}{f}{1}{2}  \\
        &&& \sdinvisn{i1}{2}  \\
        &&& \sdinvisn{i2}{2}  \\[1em]
        & \sdinvisn[yshift=1mm]{i3}{2} && \\[1em]
        \sdstoch{c1}{c}{2}{1} && \sdstoch{c2}{c}{2}{1} & \\
        \sdout{outY1}{Y} && \sdout{outY2}{Y} & \sdout[1mm]{outX}{X} \\
    }{
        \sdblackdotat{i1 out1}
        \sdblackdotat[cp2]{i2 out2}
        \sdblackdotat[cp4]{i3 out1}
        \sdblackdotat[cp5]{i3 out2}
        \sdconnect{inA out1}{cp1 in1}
        \sdconnect{cp1}{f in1}
        \sdconnect{cp1 out1}{cp4 in0}
        \sdconnect{f out2}{i2 in2}
        \sdconnect{f out1}{i1 in1}
        \sdconnect{cp2}{cp5 in1}
        \sdconnect{cp2 out1}{outX in1}
        \sdconnect{c1 out1}{outY1 in1}
        \sdconnect{c2 out1}{outY2 in1}
        \sdconnect{cp4}{c1 in1}
        \sdconnect{cp5}{c1 in2}
        \sdconnect{cp4}{c2 in1}
        \sdconnect{cp5}{c2 in2}
    }.
\end{equation}
If $f\colon A\to X\otimes Y$ is known to be deterministic given $X$ we write it as $
\stringdiagram{
    \sdinvis{i1} \\
    \sdhalfdet{f}{f}{1}{2} \\
    \sdinvisn{i2}{2} \\
    } {
        \sdconnect{i1}{f in1}
        \sdconnect{f out1}{i2 in1}
        \sdconnect{f out2}{i2 in2}
    }
$.
\end{definition}

In \cite{fritz2020-representable} it is shown that if \cref{deterministic-given-x-condition} holds for one conditional of $f$ then it holds for all conditionals, so that this definition is independent of the choice of conditional $c$.

For both \oc{BorelStoch} and \oc{Dist}, if \cref{deterministic-given-x-condition} holds then $c$ is \inlineMarginal{f}-generalised-almost-surely equal to a deterministic morphism (\cite{fritz2020-representable}, example 6.12), so for most purposes it will not hurt to think of such conditionals as genuinely deterministic, though only defined up to \inlineMarginal{f}-g.a.s.\ equivalence.

\renewcommand{\diamond}{{\,\diamondsuit}}

\begin{definition}
    \label{strongly-representable-definition}
    [Strongly representable Markov category; \cite{fritz2020-representable}, definition 6.7]
    A strongly representable Markov category is a representable Markov category in which for every morphism $f\colon A\to X\otimes Y$ there is a unique morphism $f^\diamond\colon A\to X\otimes PY$ such that (i) $f^\diamond$ is deterministic given $X$, and (ii)
    \begin{equation}
    \label{strongly-representable-condition}
    \stringdiagram{
        &[-1mm]\sdin{inA}{A}&[-1mm] \\
        &\sdstoch{f}{f}{1}{2} & \\[1em]
        \sdout{outY}{Y} && \sdout{outX}{X} \\
    }{
        \sdconnect{inA out1}{f in1}
        \sdconnect{f out1}{outY in1}
        \sdconnect{f out2}{outX in1}
    }
    =
    \stringdiagram{
        &[-1mm]\sdin{inA}{A}&[-1mm] \\
        &\sdhalfdet{f}{f^\diamond}{1}{2} & \\[1em]
        \sdwhitedot{samp} && \sdinvis{i1} \\
        \sdinvis{i3} && \sdinvis{i2} \\
        \sdout{outY}{Y} && \sdout{outX}{X} \\
    }{
        \sdconnect{inA out1}{f in1}
        \sdconnect{f out2}{i1 in1}
        \sdconnect[\,\,\,PY]{f out1}{samp in1}
        \sdconnect{i1 out1}{i2 in1}
        \sdconnect{samp out1}{i3 in1}
        \sdconnect{i2 out1}{outX in1}
        \sdconnect{i3 out1}{outY in1}
    }    .
    \end{equation}
\end{definition}
(This definition is less efficient than the one given in \cite{fritz2020-representable}, which doesn't include an assumption that the category is representable, since this can be proven from weaker assumptions.)
A strongly representable Markov category necessarily has conditionals, because $f^\diamond$ has a conditional by the definition of deterministic given $X$, and if $c\colon X\times A\to PY$ is such a conditional then $c\comp \opn{samp}_Y$ is a conditional of~$f$.

\oc{BorelStoch} is shown to be strongly representable in example 6.12 of \cite{fritz2020-representable}.
For completeness we provide a proof that \oc{Dist} is strongly representable in \cref{dist-is-strongly-representable}, where we also give an explicit construction for $f^\diamond$ in \oc{Dist}.

As a consequence of \cref{strongly-representable-definition}, in a strongly representable Markov category, if we have morphisms $f,g\colon A\to X\otimes PY$ that are known to be deterministic given $X$ then
\begin{equation}
    \label{samp-can-be-cancelled}
    \stringdiagram{
        \sdin[1mm]{inA}{A} \\
        \sdhalfdet[yshift=1mm]{f}{f}{1}{2} \\[1em]
        \sdwhitedot{samp} & \sdinvis{i1} \\
        \sdout{outY}{Y} & \sdout{outX}{X} \\
    }{
        \sdconnect{inA out1}{f in1}
        \sdconnectlabel{PY}{below=1mm}{f out1}{samp in1}
        \sdconnect{samp out1}{outY in1}
        \sdconnect{f out2}{i1 in1}
        \sdconnect{i1 out1}{outX in1}
    }
    =
    \stringdiagram{
        \sdin[1mm]{inA}{A} \\
        \sdhalfdet[yshift=1mm]{f}{g}{1}{2} \\[1em]
        \sdwhitedot{samp} & \sdinvis{i1} \\
        \sdout{outY}{Y} & \sdout{outX}{X} \\
    }{
        \sdconnect{inA out1}{f in1}
        \sdconnectlabel{PY}{below=1mm}{f out1}{samp in1}
        \sdconnect{samp out1}{outY in1}
        \sdconnect{f out2}{i1 in1}
        \sdconnect{i1 out1}{outX in1}
    }
    \qquad\Longrightarrow\qquad 
    \stringdiagram{
        \sdin[1mm]{inA}{A} \\
        \sdhalfdet[yshift=1mm]{f}{f}{1}{2} \\
        \sdinout{outY}{PY} & \sdinout{outX}{X} \\
    }{
        \sdconnect{inA out1}{f in1}
        \sdconnect{f out1}{outY in1}
        \sdconnect{f out2}{outX in1}
    }
    =
    \stringdiagram{
        \sdin[1mm]{inA}{A} \\
        \sdhalfdet[yshift=1mm]{f}{g}{1}{2} \\
        \sdinout{outY}{PY} & \sdinout{outX}{X} \\
    }{
        \sdconnect{inA out1}{f in1}
        \sdconnect{f out1}{outY in1}
        \sdconnect{f out2}{outX in1}
    }.
\end{equation}
This will be used in the proofs of \cref{b-is-a-functor,adjunction-result}.
The condition that $f$ and $g$ are deterministic given $X$ is needed to cancel sampling maps in this way, since sampling maps are not usually epimorphisms.

\section{Machines and Bayesian Filtering}
\label{main-section}

The following definitions are all relative to a Markov category $(\cat{C},\otimes,1)$ and a choice of objects called the \emph{input space} $I$ and \emph{output space} $O$, which we will assume to be fixed throughout this section.

For most of the following we will work with what we call ``comb machines,'' which are a generalisation of Moore machines.
However, many of the results also carry over to the case of Mealy machines, which we define first because they are simpler.
The following definition is standard:

\begin{definition}
    [Stochastic Mealy machine]
    A \emph{stochastic Mealy machine} is an object $S$ of \cat{C} called the \emph{state space}, together with a morphism $\alpha\colon I\otimes S\to O\otimes S$ in $\cat{C}$.
    A morphism of Mealy machines $(S,\alpha)\to(T,\beta)$ is a morphism $f\colon S\to T$ in \cat{C} such that $I\otimes S \mor{\alpha} O\otimes S \mor{\id_O\otimes f} O\otimes T = I\otimes S\mor{\id_I\otimes f}I\otimes T \mor{\beta} O\otimes T$. 
    The category of Mealy machines will be written $\oc{Mealy}(I,O)$.
\end{definition}

The idea is that a Mealy machine starts in some state in $S$, receives an input in $I$, and then produces an output in $O$ while simultaneously transitioning to a new state.
The output may depend on the input and may be correlated with the new state.
We don't require morphisms of Mealy machines to be deterministic.

We now briefly discuss Moore machines and their generalisation to the stochastic context.
In a Cartesian category, a Moore machine consists of a state space $S$ and two maps: a \emph{readout map} $S\to O$ and an \emph{update map} $I\times S \to S$.
An obvious way to generalise this to the stochastic case is to let both maps be stochastic, so that the update map has type $I\otimes S\to S$.
However, machines with this definition tend not to be very well behaved, and in practice other definitions tend to be used.

One way to make stochastic Moore machines well behaved is to make the readout map deterministic.
Machines of this kind can be expressed in terms of generalised lenses \cite{spivak2019-lenses}; this is the approach taken in \cite{myers-draft-systems}, for example.
Intuitively, requiring a deterministic readout map allows the update map to ``know'' what the machine's last output was, since this can be inferred from the current value of $S$.
However, for the present work we need the readout map to be stochastic, so we take a different approach:

\begin{definition}
    [Comb machine]
    \label{comb-machine-definition}
    A \emph{comb machine} in a Markov category \cat{C} is an object $S$ of \cat{C} (the {state space}), together with a morphism $\alpha\colon I\otimes S\to O\otimes S$ in $\cat{C}$ and a morphism $\alpha^\blob \colon S\to O$ such that
    \begin{equation}
        \label{comb-condition}
       \stringdiagram{
            \sdin[-1mm]{inS}{S} & \sdin{inI}{I} \\
            \sdstoch{a}{\alpha}{2}{2} & \\
            \sdinvisn{i1}{2} & \sdinvis{i2} \\
             & \sdout{outO}{O} \\
        }{
            \sdblackdotat{i1 in1}
            \sdconnect{inS out1}{a in1}
            \sdconnect{inI out1}{a in2}
            \sdconnect{a out1}{i1 in1}
            \sdconnect{a out2}{outO in1}
        }
        = 
    \stringdiagram{
        \sdin{inS}{S} & \sdin{inI}{I} \\
        \sdstoch{a1}{\alpha^\blob}{1}{1} & \sdblackdot{i1} \\[1em]
         & \sdout{outO}{O} \\        
    }{
            \sdconnect{inS out1}{a1 in1}
            \sdconnect{inI out1}{i1}
            \sdconnect{a1 out1}{outO in1}
    }.
    \end{equation}
    A morphism of comb machines $(S,\alpha)\to(T,\beta)$ is a morphism $f\colon S\to T$ in \cat{C} 
    that commutes with $\alpha$ and $\beta$ in the same way as for a Mealy machine.
    %such that $I\otimes S \mor{\alpha} O\otimes S \mor{\id_O\otimes f} O\otimes T = I\otimes S\mor{\id_I\otimes f}I\otimes T \mor{\beta} O\otimes T$. 
    The category of comb machines will be written $\oc{CombMachine}(I,O)$.
\end{definition}

Comb machines take their name from comb elements, as defined in a probability context in \cite{jacobs2019-surgery}.
We avoid calling them Moore machines in order to avoid confusion with the more usual definition.

\Cref{comb-condition} expresses the idea that the output of a comb machine cannot directly depend on the input.
Consequently a comb machine $\alpha$ could be seen as a Mealy machine that obeys an extra condition, namely the existence of $\alpha^\blob$ such that \cref{comb-condition} holds.
However, we will often think of them differently.
If \cat{C} has conditionals then a comb machine $\alpha$ can always be factored as
\begin{equation}
        \label{causal-conditional}
        \stringdiagram{
            \sdin[-1mm]{inS}{S} & \sdin{inI}{I} \\
            \sdstoch{u}{\alpha}{2}{2} & \\
            \sdout[-1mm]{outS}{S} & \sdout{outO}{O} \\
        }{
            \sdconnect{inS out1}{u in1}
            \sdconnect{inI out1}{u in2}
            \sdconnect{u out1}{outS in1}
            \sdconnect{u out2}{outO in1}
        }
        =
        \stringdiagram{
            \sdin{inA}{I\otimes S} & \\
            \sdblackdot{cp1} & \\
            & \sdstoch{f}{\alpha}{1}{2}  \\
            & \sdinvisn{i1}{2}  \\
            & \sdinvisn{i2}{2}  \\
            \sdstoch[yshift=1mm]{c}{u}{2}{2} & \\
            \sdout{outY}{S} & \sdout[1mm]{outX}{O} \\
        }{
            \sdblackdotat{i1 out1}
            \sdblackdotat[cp2]{i2 out2}
            \sdconnect{inA out1}{cp1 in1}
            \sdconnect{cp1}{f in1}
            \sdconnect{cp1 out1}{c in1}
            \sdconnect{f out2}{i2 in2}
            \sdconnect{f out1}{i1 in1}
            \sdconnect{cp2}{c in2}
            \sdconnect{cp2 out1}{outX in1}
            \sdconnect{c out1}{outY in1}
        }
        =
        \stringdiagram{
            \sdin{inS}{S} & \sdin{inI}{I} & \\
            \sdblackdot{cpy1} && \\
            && \sdstoch{u}{\alpha^\blob}{1}{1}  \\
            &  \sdinvis[xshift=-4mm]{i1} & \sdblackdot{cpy2} \\[1ex]
            \sdstoch[yshift=1mm]{c}{u}{3}{3} && \\
            \sdout{outS}{S} && \sdout{outO}{O} \\
        }{
            \sdconnect{inS out1}{cpy1 in1}
            \sdconnect{cpy1}{u in1}
            \sdconnect{cpy1 out1}{c in1}
            \sdconnect{u out1}{cpy2 in1}
            \sdconnect{cpy2}{c in3}
            \sdconnect{inI out1}{i1 in1}
            \sdconnect{i1 out1}{c in2}
            \sdconnect{c out1}{outS in1}
            \sdconnect{cpy2 out1}{outO in1}
        },
\end{equation}
where $u$ is a conditional of $\alpha$ as shown. 
We refer to $\alpha^\blob$ as the \emph{readout map} and $u$ as an \emph{update map} of the comb machine $(S,\alpha)$, analogously to the maps that define a Moore machine.
If $f\colon (S,\alpha)\to(T,\beta)$ is a morphism of comb machines, then we have $f\comp \alpha^\blob = \beta^\blob$, which we show in \cref{morphisms-of-comb-machines}.

The readout map $\alpha^\blob$ has the same type as in a Moore machine, $S\to O$, but the update map has type $O\otimes I\otimes S\to S$ and is only defined up to $\alpha^\blob$-g.a.s.\ equality.
This allows the next state and the output to be correlated for a given previous state and input, while still requiring the output to be independent of the input.
Although update maps are not uniquely defined, their behaviour can only differ on measure zero subsets of the output space.
In \oc{Dist} this means their behaviour can differ only on outputs $o\in O$ that cannot occur at all in a given state, i.e.\ for which $\alpha^\blob(o\mid s) = 0$.

We think of comb machines as giving their output first and then receiving their input, in contrast to Mealy machines, which first receive an input and then give an output.%
\footnote{%
This raises the question of whether we can interpose some other morphism in between $\alpha^\blob$ and $u$, so that the machine receives an input that can depend on its output, and perhaps also on the outputs of other machines.
Answering this in the most general case is rather involved and we will not address it in this paper.
However, in the case where $\cat{C}$ is \oc{FinStoch}, \cite{jacobs2019-surgery} provides a way to compose 2-combs, of which comb machines are a special case.
}
The picture to have in mind for a comb machine is this:
\begin{equation}
        \stringdiagram{
            \sdin[-1mm]{inS}{S} & \sdin{inI}{I} \\
            \sdstoch{u}{\alpha}{2}{2} & \\
            \sdout[-1mm]{outS}{S} & \sdout{outO}{O} \\
        }{
            \sdconnect{inS out1}{u in1}
            \sdconnect{inI out1}{u in2}
            \sdconnect{u out1}{outS in1}
            \sdconnect{u out2}{outO in1}
        }
    =
    \stringdiagram{
        \sdin[-1mm]{inS}{S} & \sdin{inI}{I} \\
        \sdstoch{a1}{\alpha}{2}{2} & \sdinvis{i1} \\[1em]
        \sdstoch{a2}{}{2}{2} & \sdinvis{i2} \\
        \sdout[-1mm]{outS}{S} & \sdout{outO}{O} \\        
    }{
            \sdconnect{inS out1}{a1 in1}
            \sdconnect{a1 out1}{outS in1}
            \sdconnect{inI out1}{i1}
            \sdconnect{a1 out2}{i2 in1}
            \sdconnect{i1 out1}{a2 in2}
            \sdconnect{i2 out1}{outO in1}
            \sdcomb{a1}{a2}
    },
    \qquad
    \stringdiagram{
        \sdin[-1mm]{inS}{S} & \sdin{inI}{I} \\
        \sdstoch{a1}{\alpha}{2}{2} & \sdinvis{i1} \\[1em]
        \sdstoch{a2}{}{2}{2} & \sdinvis{i2} \\
        \sdblackdot{outS} & \sdout{outO}{O} \\        
    }{
            \sdconnect{inS out1}{a1 in1}
            \sdconnect{a1 out1}{outS in1}
            \sdconnect{inI out1}{i1}
            \sdconnect{a1 out2}{i2 in1}
            \sdconnect{i1 out1}{a2 in2}
            \sdconnect{i2 out1}{outO in1}
            \sdcomb{a1}{a2}
    }
    =
    \stringdiagram{
        \sdin{inS}{S} & \sdin{inI}{I} \\
        \sdstoch{a1}{\alpha^\blob}{1}{1} & \sdblackdot{i1} \\[1em]
         & \sdout{outO}{O} \\        
    }{
            \sdconnect{inS out1}{a1 in1}
            \sdconnect{inI out1}{i1}
            \sdconnect{a1 out1}{outO in1}
    }.
\end{equation}

We now introduce the concept of a \emph{unifilar} machine.
A unifilar machine has a stochastic readout map but a deterministic update map.
(Or at least, a generalised-almost-surely deterministic one.)
The term ``unifilar'' comes from the literature on computational mechanics, where it can be used to define $\Epsilon$-machines \cite{travers2011-equivalence}.
In particular it appears in a machine-like context in \cite{barnett2015-epsilon}, proposition 5.
The formal context is different, in part because we don't assume stationarity or irreducibility, but our definition achieves the same idea.
We define unifilar machines in Mealy machine and comb machine flavours:

\defFootnote{epsilon-machines}{The concept of unifilarity appears in the definition of $\Epsilon$-machines: an $\Epsilon$-machine can be defined as the minimal unifilar representation of a stationary process \cite{travers2011-equivalence}.
We do not explore stationary processes or minimality in the present work, however.}

\begin{definition}
    [unifilar Mealy machine]
    \label{unifilar-mealy-machine-definition}
    A \emph{unifilar Mealy machine} in a Markov category is a Mealy machine $(S,\alpha)$ with the condition that $\alpha$ is deterministic given $O$.
    Additionally, we require morphisms of unifilar mealy machines to be deterministic.
    The category of unifilar Mealy machines will be written as $\oc{UnifilarMealy}(I,O)$.
\end{definition}

\begin{definition}
    [unifilar comb machine]
    \label{unifilar-comb-machine-definition}
    A \emph{unifilar comb machine} in a Markov category is a comb machine $(S,\alpha,\alpha^\blob)$ with the condition that $\alpha$ is deterministic given $O$.
    As with unifilar Mealy machines, we require morphisms of unifilar comb machines to be deterministic.
    The category of unifilar comb machines will be written as $\oc{UnifilarComb}(I,O)$.
\end{definition}

Note that $\alpha$ must admit a conditional $O\otimes I\otimes S\to S$ in order to satisfy either definition.

When we say ``unifilar machine'' without qualification we mean a unifilar comb machine.

The idea of a unifilar machine (of either type) is that all of the randomness comes from the choice of output.
A unifilar comb machine factors according to \cref{causal-conditional}, with the additional feature that the conditional $u$ is $\alpha^\blob$-generalised-almost-surely deterministic.
We interpret this as follows:
first the output $O$ is chosen stochastically (via $\alpha^\blob\colon S\to O$), and then the state updates $\alpha^\blob$-g.a.s.\ deterministically as a function of the output and the input.
As for comb machines in general, the behaviour of an update map is uniquely specified on all but $\alpha^\blob$-measure-zero subsets of the output space.

If \cat{C} is Cartesian then Mealy machines and unifilar Mealy machines coincide, as do comb machines and unifilar comb machines, both of which coincide with Moore machines.
So both comb machines and unifilar comb machines can claim to be a generalisation of Moore machines to the stochastic case.

It is worth saying something about the meaning of morphisms in these categories.
The following can be made formal using the machinery we introduce in \cref{transducer-section}, but for now we state it  informally.
We can think of a non-unifilar machine (of either flavour) as inducing a stochastic map from infinite sequences of inputs to infinite sequences of outputs, subject to a \emph{causality condition} that each output can only depend on inputs that were received at earlier points in time.
(Recall that for Mealy machines we consider the input to be received before the output, and vice versa for comb machines.)
We refer to this map as the machine's behaviour.
For Mealy machines and comb machines, a morphism $(S,\alpha)\to(T,\beta)$ witnesses that $\beta$ is capable of exhibiting all of the externally observable behaviours that $\alpha$ can exhibit.
Using a stochastic map makes sense because the states are unobserved and change randomly; we consider distributions over states to exhibit behaviours, as well as states themselves.

The interpretation of morphisms between unifilar machines is similar, but we require the morphisms to be deterministic.
%, echoing the generalised-almost-sure determinism condition on their update maps.
%
A morphism of unifilar machines witnesses not only that their externally observable behaviour is the same, but also that there is a mapping between their internal states that preserves this behaviour.
This makes sense conceptually because we will generally consider the state of a unifilar machine to be observable.

Our first result concerns the existence of an adjunction between the categories $\oc{CombMachine}(I,O)$ and $\oc{UnifilarComb}(I,O)$, from which Bayesian filtering arises.
A similar result holds for $\oc{Mealy}(I,O)$ and $\oc{UnifilarMealy}(I,O)$, which we will state at the end.
Its proof is largely the same.

We first note that there is a forgetful functor $F\colon \oc{UnifilarComb}(I,O)\to\oc{CombMachine}(I,O)$ that embeds unifilar comb machines into comb machines.
On objects it forgets that the machine obeys the deterministic-given-$O$ condition, and it also forgets that morphisms are deterministic.

If \cat{C} is strongly representable we can construct a functor in the opposite direction.
We first define it and then prove that it lands in $\oc{UnifilarComb}(I,O)$ and is a functor.
\begin{definition}
    \label{b-definition}
    Suppose that \cat{C} is a strongly representable Markov category.
    Then we define a putative functor $B\colon \oc{CombMachine}(I,O)\to \oc{UnifilarComb}(I,O)$.
    On objects it maps $(S,\alpha)\mapsto (PS, B\alpha)$, where $B\alpha = (\id_I\otimes\opn{samp_S}\comp\alpha)^{\!\diamond}$ is the unique morphism such that $B\alpha$ is deterministic given $O$ and
    \begin{equation}
        \label{action-of-b-on-objects}
        \stringdiagram{
            \sdin[-1mm]{inT}{PS} & \sdin{inI}{I} \\
            \sdinvisn{i1}{2} & \sdinvis{i2}\\[1em]
            \sdstoch{u}{\alpha}{2}{2} & \\[1em]
            \sdout[-1mm]{outT}{S} & \sdout{outO}{O} \\
        }{
            \sdwhitedotat{i1 out1}
            \sdconnect{inT out1}{i1 in1}
            \sdconnect{inI out1}{i2 in1}
            \sdconnect{i2 out1}{u in2}
            \sdconnect{u out1}{outT in1}
            \sdconnect{u out2}{outO in1}
            \sdconnectlabel{S}{below=1pt}{i1 out1}{u in1}
        }
        =        
        \stringdiagram{
            \sdin[-1mm]{inT}{PS} & \sdin{inI}{I} \\[1em]
            \sdhalfdet{u}{B\alpha}{2}{2} & \\[1em]
            \sdinvisn{i1}{2} & \sdinvis{i2}\\
            \sdout[-1mm]{outT}{S} & \sdout{outO}{O} \\
        }{
            \sdwhitedotat{i1 out1}
            \sdconnect{inT out1}{u in1}
            \sdconnect{inI out1}{u in2}
            \sdconnect{i1 out1}{outT in1}
            \sdconnect{u out2}{i2 in1}
            \sdconnect{i2 out1}{outO in1}
            \sdconnectlabel{PS}{below=1pt}{u out1}{i1 in1}
        }.
    \end{equation}
    On morphisms, $B$ maps a morphism of comb machines with underlying map $f\colon S\to T$ to a morphism of unifilar machines with underlying deterministic map $Pf\colon PS\to PT.$
\end{definition}

\begin{proposition}
    \label{b-is-a-functor}
    Let \cat{C} be a strongly representable Markov category. Then \cref{b-definition} yields a functor $B\colon \oc{CombMachine}(I,O)\to \oc{UnifilarComb}(I,O)$.
\end{proposition}

\begin{proof}
The mapping respects composition and identities by functoriality of $P$, but to prove $B$ is a functor we have to show $(i)$ that $B\alpha$ is indeed a unifilar comb machine, and $(ii)$ that that $Pf$ is indeed a morphism of unifilar comb machines.
For $(i)$ we show that if \cref{comb-condition} holds for $\alpha$ then it holds for $B\alpha$:
\begin{equation}
    \label{epistemic-readout}
      \stringdiagram{
            \sdin[-1mm]{inS}{PS} & \sdin{inI}{I} \\
            \sdhalfdet{a}{B\alpha}{2}{2} & \\[1em]
            \sdinvisn{i1}{2} & \sdinvis{i2} \\
             & \sdout{outO}{O} \\
        }{
            \sdblackdotat{i1 in1}
            \sdconnect{inS out1}{a in1}
            \sdconnect{inI out1}{a in2}
            \sdconnectlabel{PS}{below=2pt}{a out1}{i1 in1}
            \sdconnect{a out2}{i2 in1}
            \sdconnect{i2 out1}{outO in1}
        }
        =
        \stringdiagram{
            \sdin[-1mm]{inT}{PS} & \sdin{inI}{I} \\
            \sdinvisn{i1}{2} & \sdinvis{i2}\\[1em]
            \sdstoch{u}{\alpha}{2}{2} & \\
            \sdblackdot[yshift=-1mm]{outT} & \\
             & \sdout{outO}{O} \\
        }{
            \sdwhitedotat{i1 out1}
            \sdconnect{inT out1}{i1 in1}
            \sdconnect{inI out1}{i2 in1}
            \sdconnect{i2 out1}{u in2}
            \sdconnect{u out1}{outT in1}
            \sdconnect{u out2}{outO in1}
            \sdconnectlabel{S}{below=1pt}{i1 out1}{u in1}
        }
        = 
    \stringdiagram{
        \sdin{inS}{PS} & \sdin{inI}{I} \\
        \sdwhitedot{w1} & \sdblackdot{i1} \\
        \sdstoch{a1}{\alpha^\blob}{1}{1} & \\[1ex]
         & \sdout{outO}{O} \\        
    }{
            \sdconnect{inS out1}{a1 in1}
            \sdconnect{inI out1}{i1}
            \sdconnect{a1 out1}{outO in1}
    }
        = 
    \stringdiagram{
        \sdin{inS}{PS} & \sdin{inI}{I} \\
        \sddet{a1}{P\alpha^\blob}{1}{1} & \sdblackdot{i1} \\
        \sdwhitedot{w1} & \\
        &\sdout{outO}{O} \\        
    }{
            \sdconnect{inS out1}{a1 in1}
            \sdconnect{inI out1}{i1}
            \sdconnect{a1 out0}{w1 in1}
            \sdconnect{w1 out1}{outO in1}
    }.
\end{equation}
For $(ii)$, since $f$ is a morphism of comb machines we have
    \begin{aligneq}
        \stringdiagram{
            \sdin{inPS}{PS} & \sdin{inI}{I} \\
            \sdwhitedot{samp} & \sdinvis{i1}\\
            \sdstoch[yshift=1mm]{a}{\alpha}{2}{2} & \\
            \sdstoch{f}{f}{1}{1} & \sdinvis{i2} \\
            \sdout{outT}{T} & \sdout{outO}{O} \\
        } {
            \sdconnect{inPS out1}{samp in1}
            \sdconnect{samp out1}{a in1}
            \sdconnect{a out1}{f in1}
            \sdconnect{f out1}{outT in1}
            \sdconnect{inI out1}{i1 in1}
            \sdconnect{i1 out1}{a in2}
            \sdconnect{a out2}{i2 in1}
            \sdconnect{i2 out1}{outO in1}
        }
        &=
        \stringdiagram{
            \sdin{inPS}{PS} & \sdin{inI}{I} \\
            \sdwhitedot{samp} & \sdinvis{i1}\\
            \sdstoch{f}{f}{1}{1} & \sdinvis{i2} \\
            \sdstoch[yshift=1mm]{a}{\beta}{2}{2} & \\
            \sdout{outT}{T} & \sdout{outO}{O} \\
        } {
            \sdconnect{inPS out1}{samp in1}
            \sdconnect{samp out1}{f in1}
            \sdconnect{f out1}{a in1}
            \sdconnect{a out1}{outT in1}
            \sdconnect{inI out1}{i2 in1}
            \sdconnect{i2 out1}{a in2}
            \sdconnect{a out2}{outO in1}
        }
        \\
        \stringdiagram{
            \sdin{inPS}{PS} & \sdin{inI}{I} \\
            \sdhalfdet[yshift=1mm]{a}{B(\alpha)}{2}{2} & \\
            \sdwhitedot{samp1} & \sdinvis{i1}\\
            \sdstoch{f}{f}{1}{1} & \sdinvis{i2} \\
            \sdout{outT}{T} & \sdout{outO}{O} \\
        } {
            \sdconnect{inPS out1}{a in1}
            \sdconnect{a out1}{samp1 in1}
            \sdconnect{samp1 out1}{f in1}
            \sdconnect{f out1}{outT in1}
            \sdconnect{inI out1}{a in2}
            \sdconnect{a out2}{i1 in1}
            \sdconnect{i1 out1}{outO in1}
        }
        &=
        \stringdiagram{
            \sdin{inPS}{PS} & \sdin{inI}{I} \\
            \sddet{f}{Pf}{1}{1} & \sdinvis{i1} \\
            \sdwhitedot{samp1} & \sdinvis{i2}\\
            \sdstoch[yshift=1mm]{a}{\beta}{2}{2} & \\
            \sdout{outT}{T} & \sdout{outO}{O} \\
        } {
            \sdconnect{inPS out1}{f in1}
            \sdconnect{f out1}{samp1 in1}
            \sdconnect{samp1 out1}{a in1}
            \sdconnect{a out1}{outT in1}
            \sdconnect{inI out1}{i2 in1}
            \sdconnect{i2 out1}{a in2}
            \sdconnect{a out2}{outO in1}
        }
        \\
        \stringdiagram{
            \sdin{inPS}{PS} & \sdin{inI}{I} \\
            \sdhalfdet[yshift=1mm]{a}{B(\alpha)}{2}{2} & \\
            \sddet{f}{Pf}{1}{1} & \sdinvis{i1} \\
            \sdwhitedot{samp1} & \sdinvis{i2}\\
            \sdout{outT}{T} & \sdout{outO}{O} \\
        } {
            \sdconnect{inPS out1}{a in1}
            \sdconnect{a out1}{f in1}
            \sdconnect{f out1}{samp1 in1}
            \sdconnect{samp1 out1}{outT in1}
            \sdconnect{inI out1}{a in2}
            \sdconnect{a out2}{i1 in1}
            \sdconnect{i1 out1}{outO in1}
        }
        &=
        \stringdiagram{
            \sdin{inPS}{PS} & \sdin{inI}{I} \\
            \sddet{f}{Pf}{1}{1} & \sdinvis{i1} \\
            \sdhalfdet[yshift=1mm]{a}{B(\beta)}{2}{2} & \\
            \sdwhitedot{samp1} & \sdinvis{i2} \\
            \sdout{outT}{T} & \sdout{outO}{O} \\
        } {
            \sdconnect{inPS out1}{f in1}
            \sdconnect{f out1}{a in1}
            \sdconnect{a out1}{samp1 in1}
            \sdconnect{samp1 out1}{outT in1}
            \sdconnect{inI out1}{i1 in1}
            \sdconnect{i1 out1}{a in2}
            \sdconnect{a out2}{i2 in1}
            \sdconnect{i2 out1}{outO in1}
        }.
    \end{aligneq}
    We can then use the defining property of a strongly representable Markov category in the form of \cref{samp-can-be-cancelled} to cancel off the sampling maps and conclude
    \begin{equation}
        \stringdiagram{
            \sdin{inPS}{PS} & \sdin{inI}{I} \\
            \sdhalfdet[yshift=1mm]{a}{B(\alpha)}{2}{2} & \\
            \sddet{f}{Pf}{1}{1} & \sdinvis{i1} \\
            \sdout{outT}{PT} & \sdout{outO}{O} \\
        } {
            \sdconnect{inPS out1}{a in1}
            \sdconnect{a out1}{f in1}
            \sdconnect{f out1}{outT in1}
            \sdconnect{inI out1}{a in2}
            \sdconnect{a out2}{i1 in1}
            \sdconnect{i1 out1}{outO in1}
        }
    =
            \stringdiagram{
            \sdin{inPS}{PS} & \sdin{inI}{I} \\
            \sddet{f}{Pf}{1}{1} & \sdinvis{i1} \\
            \sdhalfdet[yshift=1mm]{a}{B(\beta)}{2}{2} & \\
            \sdout{outT}{PT} & \sdout{outO}{O} \\
        } {
            \sdconnect{inPS out1}{f in1}
            \sdconnect{f out1}{a in1}
            \sdconnect{a out1}{outT in1}
            \sdconnect{inI out1}{i1 in1}
            \sdconnect{i1 out1}{a in2}
            \sdconnect{a out2}{outO in1}
        },
\end{equation}
i.e. $Pf$ is a morphism of unifilar machines.
\end{proof}

We think of the functor $B$ as taking a dynamical model (in the form of a comb machine) and converting it into an epistemic model in the form of a unifilar machine.
To unpack this, first consider a comb machine $(H,\kappa)$ as a dynamical model: we think of $H$ as a set of hidden states and $\kappa$ as a dynamical process that emits outputs and stochastically changes the hidden state as a function of the input.

Then $B((H,\kappa))$ is a unifilar machine.
In order to view it as an epistemic model, we write it using \cref{epistemic-readout} as
\begin{equation}
    \label{b-of-kappa-factoring}
        B \left(
        \stringdiagram{
            \sdin[-1mm]{inS}{H} & \sdin{inI}{I} \\
            \sdstoch{u}{\kappa}{2}{2} & \\
            \sdout[-1mm]{outS}{H} & \sdout{outO}{O} \\
        }{
            \sdconnect{inS out1}{u in1}
            \sdconnect{inI out1}{u in2}
            \sdconnect{u out1}{outS in1}
            \sdconnect{u out2}{outO in1}
        }
        \right)
        =
            \stringdiagram{
            \sdin{inS}{PH} &\sdin{inI}{I}&  \\
            \sdblackdot{cpy1} && \\
            && \sddet{u}{P \kappa^\blob}{1}{1}  \\
            && \sdwhitedot{samp1}  \\
            & \sdinvis[xshift=-4mm]{i1} & \sdblackdot{cpy2}  \\[1em]
            \sdstoch[yshift=1mm]{c}{u}{3}{3} && \\
            \sdout{outS}{PH} && \sdout{outO}{O}  \\
        }{
            \sdconnect{inS out1}{cpy1 in1}
            \sdconnect{cpy1}{u in1}
            \sdconnect{u out1}{samp1 in1}
            \sdconnect{cpy1 out1}{c in1}
            \sdconnect{samp1 out1}{cpy2 in1}
            \sdconnect{cpy2}{c in3}
            \sdconnect{inI out1}{i1 in1}
            \sdconnect{i1 out1}{c in2}
            \sdconnect{c out1}{outS in1}
            \sdconnect{cpy2 out1}{outO in1}
        },
\end{equation}
in which the conditional $u$ is $(P\kappa^\blob\comp\opn{samp})$-g.a.s.\ deterministic as well as $(P\kappa^\blob\comp\opn{samp})$-g.a.s.\ unique. %ly defined.
We will think of the state space $PH$ as the space of ``beliefs about $H$,'' and the update map $u$ as updating those beliefs using Bayesian filtering.

We imagine these beliefs to be
held by an idealised Bayesian reasoner, whose prior at any given time is an element of the state space, $PH$.
This Bayesian reasoner does not interact with the machine $\kappa$, it only observes the inputs that $\kappa$ receives and the outputs it emits in response, updating its prior to a posterior at each time step.

The output map $PH\mor{\opn{samp}_H}H\mor{\kappa^\blob}O = PH\mor{P\kappa^\blob}PO\mor{\opn{samp}_O}O$ ``simulates'' the output of $\kappa$.
The map $P\kappa^\blob$ maps the reasoner's prior beliefs about $H$ to its beliefs about the next output it will observe.

The update map $u$ performs Bayesian filtering.
It takes as input a probability measure over the hidden states along with an input and an output, and it returns a new probability measure over hidden states.
We think of it as taking a prior over the \emph{current} value hidden state and returning a posterior distribution over the \emph{next} value of the hidden state, conditioned on the observed output.
It thus combines Bayesian updating with ``simulating'' the stochastic change in $H$.

The output from $u$ is the posterior distribution.
It is only defined up to almost sure equivalence.
In the case where $O$ is finite this is because for a given output $o\in O$ and a given belief $b\in PH$ we might have $(b\comp P\kappa^\blob)(o) = 0$, i.e.\ the output $o$ is ``subjectively impossible'' according to the agent's current epistemic state.
In this case calculating the Bayesian posterior in the usual way would lead to a division by zero, so there is no consistent value that the posterior distribution could take.
%
%Defining Bayesian filtering in terms of unifilar machines avoids needing to define any posterior in such cases.
%[[can I say more here?]]
Since the update map $u$ is only defined up to $(P\kappa^\blob\comp\opn{samp})$-g.a.s.\ equality its output only matters in those cases where this doesn't happen.

As one would expect from a Bayesian filter, instances of the map $u$ can be chained together in such a way that, given an initial distribution over $H$ and a sequence of inputs, we can recover the posterior over $H$ for a given observed sequence of outputs.
We give a precise statement of this in \cref{filtering-on-sequences}, though we omit the proof for reasons of space.

We can thus regard the functor $B$ as taking a dynamical model as input and turning it into an epistemic model.
%
%process of turning a dynamical model into an epistemic model 
We remark that a similar operation is performed in the process of solving a partially observable Markov decision process (POMDP).
A POMDP consists of some kind of machine --- for simplicity let us say a comb machine $(H,\kappa)$ --- together with a reward function.
This machine is a dynamical model of some environment, and the goal is to find a ``policy'' that maximises the expected amount of reward that is accumulated over time, usually with an exponential discounting factor.
(We will not consider reward functions in the present work.)
A common solution technique involves converting the POMDP into a Markov decision process (MDP), which is a simpler class of problem.
In an MDP the state space is assumed to be fully observed, so that there is no need to consider outputs.
In an MDP the machine only takes inputs, and changes state stochastically as a function of its input, so it can be seen as an object of $\oc{CombMachine}(I,1)$.
%
%(The idea is that in an MDP the state is directly observable, so no output is needed.)
%
Again there is an associated reward function, which we will not consider in detail.
To turn a POMDP into an MDP one forms the so-called ``belief MDP'', whose state space is given by probability distributions over $H$.
In our framework it is given by
$
        \stringdiagram{
            \sdin[-1mm]{inS}{PH} & \sdin{inI}{I} \\
            \sdhalfdet{u}{B(\kappa)}{2}{2} & \\[1em]
            & \sdblackdot{del1} \\
            \sdout[-1mm]{outS}{PH} & \\
        }{
            \sdconnect{inS out1}{u in1}
            \sdconnect{inI out1}{u in2}
            \sdconnect{u out1}{outS in1}
            \sdconnect[O]{u out2}{del1 in1}
        }.
$
Note that this is a stochastic map in general.
For an approach to POMDPs that is closely related to the present work, see \cite{biehl2023-interpreting}.

The following is our main technical result.

\begin{theorem}
    \label{adjunction-result}
    When \cat{C} is a strongly representable Markov category, the functor $B$ is right adjoint to the forgetful functor $F$,
    \begin{equation*}
        \begin{tikzcd}
            \oc{CombMachine}(I,O)\arrow[r, leftarrow, shift left=1ex, "F"{name=G}] & \oc{UnifilarComb}(I,O)\arrow[l, leftarrow, shift left=.5ex, "B"{name=F}]
            \arrow[phantom, from=F, to=G, , "\scriptscriptstyle\boldsymbol{\bot}"].
        \end{tikzcd}
% % 
\end{equation*}
\end{theorem}
\begin{proof}
    We show that if $f\colon S\to H$ is the map in $\cat{C}$ underlying a morphism $F((S,\alpha)) \to (H,\kappa)$ in $\oc{Comb}$- $\oc{Machine}(I,O)$ then $f^\squ\colon S\to PH$ is the deterministic map underlying a morphism $(S,\alpha)\to B((H,\kappa))$ in $\oc{UnifilarComb}(I,O)$, and vice versa.
    This will form the natural isomorphism of hom-sets needed for an adjunction.

    Suppose $f\colon S\to H$ is the map underlying a morphism $F((S,\alpha)) \to (H,\kappa)$ in $\oc{CombMachine}(I,O)$.
    Then we have the following (where, as always, all diagrams are in \cat{C}):
    \begin{aligneq}
        \stringdiagram{
            \sdin{inS}{S} & \sdin{inI}{I} \\
            \sdhalfdet[yshift=1mm]{a}{F(\alpha)}{2}{2} & \\
            \sdstoch{f}{f}{1}{1} & \sdinvis{i1} \\
            \sdout{outH}{H} & \sdout{outO}{O} \\
        }{
            % bottom wire
            \sdconnect{inS}{a in1}
            \sdconnect{a out1}{f in1}
            \sdconnect{f out1}{outH in1}
            % top wire
            \sdconnect{inI out1}{a in2}
            \sdconnect{a out2}{i1 in1}
            \sdconnect{i1 out1}{outO in1}
        }
        &=
        \stringdiagram{
            \sdin{inS}{S} & \sdin{inI}{I} \\
            \sdstoch{f}{f}{1}{1} & \sdinvis{i1} \\
            \sdstoch[yshift=1mm]{a}{\kappa}{2}{2} & \\
            \sdout{outH}{H} & \sdout{outO}{O} \\
        }{
            % bottom wire
            \sdconnect{inS}{f in1}
            \sdconnect{f out1}{a in1}
            \sdconnect{a out1}{outH in1}
            % top wire
            \sdconnect{inI out1}{i1 in1}
            \sdconnect{i1 out1}{a in2}
            \sdconnect{a out2}{outO in1}
        }
        \\[1ex]
        \stringdiagram{
            \sdin{inS}{S} & \sdin{inI}{I} \\
            \sdhalfdet[yshift=1mm]{a}{F(\alpha)}{2}{2} & \\
            \sddet{f}{f^\squ}{1}{1} & \sdinvis{i1} \\
            \sdwhitedot{samp1} & \\
            \sdout{outH}{H} & \sdout{outO}{O} \\
        }{
            % bottom wire
            \sdconnect{inS}{a in1}
            \sdconnect{a out1}{f in1}
            \sdconnect{f out1}{outH in1}
            % top wire
            \sdconnect{inI out1}{a in2}
            \sdconnect{a out2}{i1 in1}
            \sdconnect{i1 out1}{outO in1}
        }
        &=
        \stringdiagram{
            \sdin{inS}{S} & \sdin{inI}{I} \\
            \sddet{f}{f^\squ}{1}{1} &  \\
            \sdwhitedot{samp1} & \sdinvis{i1}\\
            \sdstoch[yshift=1mm]{a}{\kappa}{2}{2} & \\
            \sdout{outH}{H} & \sdout{outO}{O} \\
        }{
            % bottom wire
            \sdconnect{inS}{f in1}
            \sdconnect{f out1}{a in1}
            \sdconnect{a out1}{outH in1}
            % top wire
            \sdconnect{inI out1}{i1 in1}
            \sdconnect{i1 out1}{a in2}
            \sdconnect{a out2}{outO in1}
        }        
        \\[1ex]
        \stringdiagram{
            \sdin{inS}{S} & \sdin{inI}{I} \\
            \sdhalfdet[yshift=1mm]{a}{\alpha}{2}{2} & \\
            \sddet{f}{f^\squ}{1}{1} & \sdinvis{i1} \\
            \sdwhitedot{samp1} & \\
            \sdout{outH}{H} & \sdout{outO}{O} \\
        }{
            % bottom wire
            \sdconnect{inS}{a in1}
            \sdconnect{a out1}{f in1}
            \sdconnect{f out1}{outH in1}
            % top wire
            \sdconnect{inI out1}{a in2}
            \sdconnect{a out2}{i1 in1}
            \sdconnect{i1 out1}{outO in1}
        }
        &=
        \stringdiagram{
            \sdin{inS}{S} & \sdin{inI}{I} \\
            \sddet{f}{f^\squ}{1}{1} & \sdinvis{i1} \\
            \sdhalfdet[yshift=1mm]{a}{B(\kappa)}{2}{2} & \\
            \sdwhitedot{samp1} & \sdinvis{i2}\\
            \sdout{outH}{H} & \sdout{outO}{O} \\
        }{
            % bottom wire
            \sdconnect{inS}{f in1}
            \sdconnect{f out1}{a in1}
            \sdconnect{a out1}{outH in1}
            % top wire
            \sdconnect{inI out1}{i1 in1}
            \sdconnect{i1 out1}{a in2}
            \sdconnect{a out2}{i2 in1}
            \sdconnect{i2 out1}{outO in1}
        }.
    \end{aligneq}
    We can then use representability of \cat{C} in the form of \cref{samp-can-be-cancelled} again to conclude that
    \begin{equation}
                \stringdiagram{
            \sdin{inS}{S} & \sdin{inI}{I} \\
            \sdhalfdet[yshift=1mm]{a}{\alpha}{2}{2} & \\
            \sddet{f}{f^\squ}{1}{1} & \sdinvis{i1} \\
            \sdout{outH}{H} & \sdout{outO}{O} \\
        }{
            % bottom wire
            \sdconnect{inS}{a in1}
            \sdconnect{a out1}{f in1}
            \sdconnect{f out1}{outH in1}
            % top wire
            \sdconnect{inI out1}{a in2}
            \sdconnect{a out2}{i1 in1}
            \sdconnect{i1 out1}{outO in1}
        }
        =
        \stringdiagram{
            \sdin{inS}{S} & \sdin{inI}{I} \\
            \sddet{f}{f^\squ}{1}{1} & \sdinvis{i1} \\
            \sdhalfdet[yshift=1mm]{a}{B(\kappa)}{2}{2} & \\
            \sdout{outH}{H} & \sdout{outO}{O} \\
        }{
            % bottom wire
            \sdconnect{inS}{f in1}
            \sdconnect{f out1}{a in1}
            \sdconnect{a out1}{outH in1}
            % top wire
            \sdconnect{inI out1}{i1 in1}
            \sdconnect{i1 out1}{a in2}
            \sdconnect{a out2}{outO in1}
        },
    \end{equation}
    so that $f^\squ$ underlies a morphism $(S,\alpha)\to B((H,\kappa))$ in $\oc{UnifilarComb}(I,O)$.
    Each of these steps can be reversed, so this gives a bijection $\oc{CombMachine}(I,O)(F(\blank),\bblank) \cong \oc{UnifilarComb}(I,O)(\blank,B(\bblank))$.
    Naturality follows from the naturality of the sampling map.
\end{proof}

This adjunction is related to the one between $P\colon \cat{C}_\det\to\cat{C}$ and $\cat{C}_\det\hookrightarrow \cat{C}$ in a representable Markov category, and it shares the same unit and counit.
The unit has components $\delta_X\colon X\to PX$ and the counit has components $\opn{samp}_X\colon PX \to X$, where $PX=BX$ on objects.

The existence of this adjunction has some interesting consequences.
We have already established that the unifilar machine $B((H,\kappa))$ can be seen as an epistemic model of the comb machine $(H,\kappa)$, seen as a dynamical model.
But now consider a morphism $(S,\alpha)\to B((H,\kappa))$ in $\oc{UnifilarComb}(I,O)$ from some other unifilar machine into $B((H,\kappa))$.
We argue that when equipped with such a morphism, $(S,\alpha)$ \emph{also} deserves to be seen as modelling $(H,\kappa)$.

To see this we consider its adjoint map $F((S,\alpha))\to (H,\kappa)$, which is given by an underlying map $\psi\colon S\to H$ in $\cat{C}$ such that
\begin{equation}
        \stringdiagram{
            \sdin{inS}{S} & \sdin{inI}{I} \\
            \sdstoch{f}{\psi}{1}{1} & \sdinvis{i1} \\
            \sdstoch[yshift=1mm]{a}{\kappa}{2}{2} & \\
            \sdout{outH}{H} & \sdout{outO}{O} \\
        }{
            % bottom wire
            \sdconnect{inS}{f in1}
            \sdconnect{f out1}{a in1}
            \sdconnect{a out1}{outH in1}
            % top wire
            \sdconnect{inI out1}{i1 in1}
            \sdconnect{i1 out1}{a in2}
            \sdconnect{a out2}{outO in1}
        }
        =
        \stringdiagram{
            \sdin{inS}{S} & \sdin{inI}{I} \\
            \sdhalfdet[yshift=1mm]{a}{\alpha}{2}{2} & \\
            \sdstoch{f}{\psi}{1}{1} & \sdinvis{i1} \\
            \sdout{outH}{H} & \sdout{outO}{O} \\
        }{
            % bottom wire
            \sdconnect{inS}{a in1}
            \sdconnect{a out1}{f in1}
            \sdconnect{f out1}{outH in1}
            % top wire
            \sdconnect{inI out1}{a in2}
            \sdconnect{a out2}{i1 in1}
            \sdconnect{i1 out1}{outO in1}
        }
        ,
\end{equation}
or
\begin{equation}
        \stringdiagram{
            \sdin{inS}{S} & \sdin{inI}{I} \\
            \sdstoch{f}{\psi}{1}{1} & \sdinvis{i1} \\
            \sdstoch[yshift=1mm]{a}{\kappa}{2}{2} & \\
            \sdout{outH}{H} & \sdout{outO}{O} \\
        }{
            % bottom wire
            \sdconnect{inS}{f in1}
            \sdconnect{f out1}{a in1}
            \sdconnect{a out1}{outH in1}
            % top wire
            \sdconnect{inI out1}{i1 in1}
            \sdconnect{i1 out1}{a in2}
            \sdconnect{a out2}{outO in1}
        }
        =
        \stringdiagram{
            \sdin{inS}{S} & \sdin{inI}{I}& \\
            \sdblackdot{cpy1} && \\
            && \sdstoch{a}{\alpha^\blob}{1}{1}  \\
            & \sdinvis[xshift=-4mm]{i1} & \sdblackdot{cpy2} \\
            \sdstoch[yshift=1mm]{u}{u}{3}{3} && \\
            \sdstoch{f}{\psi}{1}{1} && \\
            \sdout{outH}{H} && \sdout{outO}{O} \\
        }{
            % % bottom wire
            \sdconnect{inS}{u in1}
            \sdconnect{u out1}{f in1}
            \sdconnect{f out1}{outH in1}
            % first branch
            \sdconnect{cpy1}{a in1}
            % middle wire
            \sdconnect{a out0}{outO in1}
            % branch downward
            \sdconnect{cpy2}{u in3}
            % % top wire
            \sdconnect{inI out1}{i1 in1}
            \sdconnect{i1 out1}{u in2}
        }
        ,
\end{equation}
where $u$ is an update map for $\alpha$.
By marginalising both sides (i.e.\ post-composing with $\id_O\otimes\opn{del}_H$) we have $\alpha^\blob = \psi\comp\kappa^\blob$, so this equation becomes
\begin{equation}
        \label{jacobs-equation-for-filtering}
        \stringdiagram{
            \sdin{inS}{S} & \sdin{inI}{I} \\
            \sdstoch{f}{\psi}{1}{1} & \sdinvis{i1} \\
            \sdstoch[yshift=1mm]{a}{\kappa}{2}{2} & \\
            \sdout{outH}{H} & \sdout{outO}{O} \\
        }{
            % bottom wire
            \sdconnect{inS}{f in1}
            \sdconnect{f out1}{a in1}
            \sdconnect{a out1}{outH in1}
            % top wire
            \sdconnect{inI out1}{i1 in1}
            \sdconnect{i1 out1}{a in2}
            \sdconnect{a out2}{outO in1}
        }
        =
        \stringdiagram{
            \sdin{inS}{S} & \sdin{inI}{I} & \\
            \sdblackdot{cpy1} && \\
            && \sdstoch{a}{\psi}{1}{1}  \\
            && \sdstoch{k}{\kappa^\blob}{1}{1}  \\
            & \sdinvis[xshift=-4mm]{i1} & \sdblackdot{cpy2} \\
            \sdstoch[yshift=1mm]{u}{u}{3}{3} && \\
            \sdstoch{f}{\psi}{1}{1} && \\
            \sdout{outH}{H} && \sdout{outO}{O} \\
        }{
            % % bottom wire
            \sdconnect{inS}{u in1}
            \sdconnect{u out1}{f in1}
            \sdconnect{f out1}{outH in1}
            % first branch
            \sdconnect{cpy1}{a in1}
            % middle wire
            \sdconnect{a out0}{outO in1}
            % branch downward
            \sdconnect{cpy2}{u in3}
            % % top wire
            \sdconnect{inI out1}{i1 in1}
            \sdconnect{i1 out1}{u in2}
        }
        ,
\end{equation}
where $u$ is $\psi\comp\kappa^\blob$-g.a.s.\ deterministic.
This is a Bayesian filtering version of Jacobs' \cite{jacobs2020-conjugate} definition of conjugate priors.
It is not quite the same as the one in \cite{virgo2021-interpreting} because in that paper $u$ is not assumed to be almost-surely deterministic, so a stronger equation is needed.
However, it is conceptually the same.

The morphism $\psi$ can be regarded as what the author and colleagues called an \emph{interpretation map} in~\cite{virgo2021-interpreting}.
This means we think of the update map $u$ as a physical machine whose job is to keep track of an epistemic model of $\kappa$.
At each step it receives both the input that was given to $\kappa$ and the output that $\kappa$ emitted in response.
The machine's physical state ($S$) then updates in a ($\psi\comp\kappa^\blob$-g.a.s.) deterministic way.

\Cref{jacobs-equation-for-filtering} expresses the idea that when the machine receives a new piece of information in the form of an $(i,o)$ pair it should update its beliefs in a consistent way.
The left-hand side can be seen as the agent's current beliefs about the {next} output and the {next} value of the hidden state, as a function of the next input.
The equation says that after receiving an input and output pair, its new beliefs about the {current} hidden state should equal a conditional of its prior beliefs, conditioned on $i$ and $o$.

The adjoint map $\psi^\squ\colon S\to PH$ can then be seen as mapping the unifilar machine's physical state to a probability measure over $H$ that we think of as ``the machine's beliefs about $H$,'' i.e.\ its current Bayesian prior.
Since $\psi^\squ$ underlies a morphism $(S,\alpha)\to B((H,\kappa))$ it means that $\alpha$'s updates have to be able to `simulate' the idealised Bayesian filtering that $B((H,\kappa))$ performs.
The map $\psi^\squ$ can thus be seen as assigning a semantic {meaning} to the states of the unifilar machine.

We now state the corresponding result for Mealy machines: as for comb machines there are functors $\oc{Mealy}(I,O) \mathrel{\underset{F}{\overset{B}{\smash[b]{\rightleftarrows}}}} \oc{UnifilarMealy}(I,O)$ such that $F$ is left adjoint to $B$.
The definitions and proofs are the same as for comb machines and unifilar comb machines, except that we don't need to care about the comb condition.
These functors can be thought of in the same terms, with $B$ mapping a dynamical model to a corresponding epistemic model.
The Mealy machine version of \cref{jacobs-equation-for-filtering} is
\begin{equation}
        \label{jacobs-equation-for-filtering-mealy-version}
        \stringdiagram{
            \sdin{inS}{S} & \sdin{inI}{I} \\
            \sdstoch{f}{\psi}{1}{1} & \sdinvis{i1} \\
            \sdstoch[yshift=1mm]{a}{\kappa}{2}{2} & \\
            \sdout{outH}{H} & \sdout{outO}{O} \\
        }{
            % bottom wire
            \sdconnect{inS}{f in1}
            \sdconnect{f out1}{a in1}
            \sdconnect{a out1}{outH in1}
            % top wire
            \sdconnect{inI out1}{i1 in1}
            \sdconnect{i1 out1}{a in2}
            \sdconnect{a out2}{outO in1}
        }
        =
        \stringdiagram{
            \sdin{inS}{S} && \sdin{inI}{I} \\
            \sdblackdot{cpy1} && \\
            & \sdstoch{a}{\psi}{1}{1} & \\
            && \sdblackdot{cpy3} \\
            & \sdstoch[yshift=1mm]{k}{\kappa}{2}{2} & \\
            & \sdblackdot{del1} & \sdinvis{i1} \\[1em]
            & \sdblackdot[yshift=2mm,xshift=-2mm]{cpy2} &  \\
            \sdstoch[yshift=1mm]{u}{u}{3}{3} && \\
            \sdstoch{f}{\psi}{1}{1} && \\
            \sdout{outH}{H} & \sdout[2mm]{outO}{O} \\
        }{
            % % bottom wire
            \sdconnect{inS}{u in1}
            \sdconnect{u out1}{f in1}
            \sdconnect{f out1}{outH in1}
            % first branch
            \sdconnect{cpy1}{a in1}
            % middle wire
            \sdconnectlabel{H}{below=2pt}{a out0}{k in1}
            \sdconnect{k out2}{outO in1}
            \sdconnectlabel{H}{below=2pt}{k out1}{del1 in1}
            % branch from top to middle
            \sdconnect{cpy3}{k in2}
            % branch downward
            \sdconnect{cpy2}{u in3}
            % % top wire
            \sdconnect{inI out1}{i1 in1}
            \sdconnect{i1 out1}{u in2}
        }.
\end{equation}

An example with enormous practical importance in control theory is the Kalman filter.
Although Kalman \cite{kalman1960-filtering} originally derived it in terms of error minimisation it is well known that it can also be constructed as a Bayesian filter.
(See \cite{gurajala2021-kalman} for a somewhat informal exposition, for example.)

Here we give only the briefest sketch of how the Kalman filter can be formulated in our framework.
We consider a version with measurement noise but no input signal.
Unlike most treatments we allow the measurement noise and the process noise to be correlated.

\begin{example}[Kalman filter]
    \label{kalman-filter-example}
    Let $\cat{C} = \oc{BorelStoch}$ and let $I=1, H=\R^n, O=\R^m$.
    Consider a comb machine $(H,\kappa)$ where $\kappa(\blank\mid h)$ is normally distributed according to $\kappa(\blank\mid h) \sim \mathscr{N}(Ah, \Sigma)$, where $A$ is an $(m+n)\times   n$ matrix and the $(m+n)\times (m+n)$ covariance matrix $\Sigma$ doesn't depend on $h$.
    Taking $(H,\kappa)$ as a dynamical model of a process, we want to construct a unfilar comb machine that will act as an epistemic model.
    
    To do this, we first note that if $p\colon 1\to H$ is a normal distribution with mean $\bar h$ and covariance matrix $\Sigma_p$, then 
    $
    \stringdiagram{
        \sdstoch{p}{p}{1}{2} \\
        \sdstoch{k}{\kappa}{2}{2} \\
        \sdinvisn{out}{2} \\
    }{
        \sdconnect{p out1}{k in1}
        \sdconnect{k out1}{out in1}
        \sdconnect{k out2}{out in2}
    }
    $
    is also Gaussian, with mean $\bar s$ and covariance $\Sigma'\coloneqq A\Sigma_p A^T + \Sigma$.
    (See section 6 of \cite{fritz2020-markov}, for example.)
    Writing $\Sigma'$ in block form as $\Sigma' = \big(\begin{smallmatrix}
        \Sigma'_{OO} & \Sigma'_{OH} \\
        \Sigma'_{HO} & \Sigma'_{HH} \\
    \end{smallmatrix}\big)$, this
    distribution $p\comp\kappa$ has a conditional $c\colon O\to H$ given by
    \begin{equation}
    \label{gaussian-conditional}
    c(\blank\mid o) \sim \mathscr{N}(\Sigma'_{HO}\Sigma'^-_{OO} o \,,\, \Sigma'_{HH} - \Sigma'_{HO}\Sigma'^-_{OO}\Sigma'_{OH}),
    \end{equation}
    where $\Sigma'^-_{OO}$ is the Moore-Penrose pseudoinverse of $\Sigma'_{OO}$.
    (See example 11.8 of \cite{fritz2020-markov}.)

    Let us therefore define a unifilar machine $(S,\alpha)$ where $S$ is the set of pairs $(\bar h, \Sigma_p)$, where $\bar h\in \R^n$ and $\Sigma_p$ is an $n\times n$ positive definite matrix.
    To define $\alpha\colon S\to O\times S$ we first define the map $\psi\colon S\to H$, which maps $(\bar h, \Sigma_p)$ to a Gaussian with mean $\bar h$ and covariance $\Sigma_p$.
    We can define $\alpha\colon S\to O\times S$ by the readout function $
    \stringdiagram{
        \sdin{in}{S} \\
        \sdstoch{a}{\alpha^\blob}{1}{1} \\
        \sdout{out}{O} \\
    } {
        \sdconnect{in}{out}
    }
    =
    \stringdiagram {
        \sdinvis{i1} \\
        \sdstoch{n}{\psi}{1}{1} \\
        \sdstoch{f}{\kappa}{1}{2} \\[-1.5mm]
        \sdinvisn{i2}{2} \\[-1.5mm]
        \sdinvisn{i3}{2} \\
    } {
        \sdblackdotat{i2 in1}
        \sdconnect{i1}{f in1}
        \sdconnect{f out1}{i2 in1}
        \sdconnect{f out2}{i3 in2}
    }
$,
and a deterministic update map $u\colon O\times S\to S$ that maps $((\bar h, \Sigma_p),o)$ to $(\Sigma'_{HO}\Sigma'^-_{OO} o \,,\, \Sigma'_{HH} - \Sigma'_{HO}\Sigma'^-_{OO}\Sigma'_{OH})$, as in \cref{gaussian-conditional}.

By construction, the maps $u$, $\psi$, $\alpha^\blob$ and $\kappa$ obey \cref{jacobs-equation-for-filtering}, and we can conclude that $\psi^\squ$ is a map of unfilar machines from $(S,\alpha)$ to $B((H,\kappa))$.

The update map $u$ is a version of the Kalman filter.
Its state space $H$ parametrises Gaussian distributions via the map $\psi^\squ$.
The machine $(H,\kappa)$ is such that for a Gaussian prior the posterior will also be a Gaussian, and therefore the deterministic map $u$ only has to update the parameters upon receiving new data.
We note a similarity between Kalman filtering and the category \oc{Gauss} defined in \cite{fritz2020-markov}, which we referred to in deriving it.
\end{example}

\subsection{Bayesian Inference and Conjugate Priors}
\label{bayesian-inference-section}

Up to now we have considered a version of Bayesian filtering in which the systems being modelled have the form of a comb machine.
In this section we consider an important special case of this, in which the system being modelled simply emits independent and identically distributed outputs.
This corresponds to the standard setting of Bayesian inference, where we receive independent samples from a known distribution with an unknown (but fixed) value for its parameters, and wish to use this data to make inferences about the parameters.

In this section we primarily consider machines whose input space is the terminal object in $\cat{C}$.
In this case the distinction between comb machines and Mealy machines isn't relevant, and we refer to such machines as generators, defining $\oc{Generator}(X) = \oc{Mealy}(1,X) \cong \oc{UnifilarMachine}(1,X)$ and $\oc{UnifilarGenerator}(X) = \oc{UnifilarMealy}(1,X) \cong \oc{UnifilarComb}(1,X)$.

To model Bayesian inference in our setup we consider objects of $\oc{Generator}(X)$ represented by morphisms in $\cat{C}$ of the following special form:
\begin{equation}
    \stringdiagram{
        \sdin{inT}{\Theta} & \\
        \sdstoch[yshift=1mm]{f}{f^\circ}{2}{2} & \\
        \sdout{outT}{\Theta} & \sdout{outX}{X} \\
    }{
        \sdconnect{inT}{f in1}
        \sdconnect{f out2}{outX in1}
        \sdconnect{f out1}{outT in1}
    }
    \coloneqq
    \stringdiagram{
        \sdin{inT}{\Theta} & \\
        \sdblackdot{cpy1} & \\
        & \sdstoch{f}{f}{1}{1} \\
        \sdout{outT}{\Theta} & \sdout{outX}{X} \\
    }{
        \sdconnect{inT}{cpy1 in1}
        \sdconnect{cpy1}{f in1}
        \sdconnect{f out1}{outX in1}
        \sdconnect{cpy1 out1}{outT in1}
    }.
\end{equation}
In this setting we call $X$ the \emph{sample space} and $\Theta$ the \emph{parameter space}, and we think of $f$ as a statistical model, that is, a family of distributions over $X$ parametrised by $\Theta$.

Applying the functor $B$ we get
\begin{equation}
    \label{bayes-for-f}
    \stringdiagram{
        \sdin{inT}{P\Theta} & \\
        \sdhalfdet[yshift=1mm]{f}{B(f^\circ)}{2}{2} & \\
        \sdout{outT}{P\Theta} & \sdout{outX}{X} \\
    }{
        \sdconnect{inT}{f in1}
        \sdconnect{f out2}{outX in1}
        \sdconnect{f out1}{outT in1}
    }
    =
    \stringdiagram{
        \sdin{inT}{P\Theta} & \\
        \sdblackdot{cpy1} & \\
        & \sddet{f}{Pf}{1}{1} \\
        & \sdwhitedot{samp1} \\
        & \sdblackdot{cpy2} & \\
        \sdstoch[yshift=1mm]{bf}{\opn{Bayes}_f}{2}{2} \\
        \sdout{outT}{P\Theta} & \sdout{outX}{X} \\
    }{
        \sdconnect{inT}{cpy1 in1}
        \sdconnect{cpy1}{f in1}
        \sdconnect{f out1}{samp1 in1}
        \sdconnect{samp1 out1}{outX in1}
        \sdconnect{cpy2}{bf in2}
        \sdconnect{cpy1 out1}{bf in1}
        \sdconnect{bf out1}{outT in1}
    },
\end{equation}
where we have called the conditional $\opn{Bayes}_f$ because that is what it does: it takes in a prior over the parameters together with some data $x\in X$, and returns the Bayesian posterior over the parameters, according to the model $f$.
We give a precise statement and proof for this claim in \cref{bayes-f-does-bayes}.

If we consider a map $\psi^\squ$ into this machine from some other unifilar machine $(S,\alpha)$, we obtain exactly the notion of a conjugate prior.
Its adjoint map of comb machines, $\psi\colon F((S,\alpha))\to f^\circ$, obeys
\begin{equation}
        \label{jacobs-equation-per-se}
        \stringdiagram{
            \sdin{inS}{S} & \\
            \sdstoch{psi}{\psi}{1}{1} & \sdinvis{i1} \\
            \sdblackdot{cpy1} & \\
            & \sdstoch{f}{f}{1}{1} \\
            \sdout{outH}{\Theta} & \sdout{outO}{X} \\
        }{
            % bottom wire
            \sdconnect{inS}{psi in1}
            \sdconnect{psi out1}{cpy1 in1}
            \sdconnect{cpy1 out1}{outH in1}
            % top wire
            \sdconnect{cpy1}{f in1}
            \sdconnect{f out1}{outO in1}
        }
        =
        \stringdiagram{
            \sdin{inS}{S} && \\
            \sdblackdot{cpy1} && \\
            & \sdstoch{a}{\psi}{1}{1} & \\
            & \sdstoch{k}{f}{1}{1} & \\
            & \sdblackdot{cpy2} & \sdinvis{i1} \\
            \sdstoch[yshift=1mm]{u}{u}{2}{2} && \\
            \sdstoch{f}{\psi}{1}{1} && \\
            \sdout{outH}{\Theta} & \sdout{outO}{X} \\
        }{
            % % bottom wire
            \sdconnect{inS}{u in1}
            \sdconnect{u out1}{f in1}
            \sdconnect{f out1}{outH in1}
            % first branch
            \sdconnect{cpy1}{a in1}
            % middle wire
            \sdconnect{a out0}{outO in1}
            % branch downward
            \sdconnect{cpy2}{u in2}
        }
        ,
\end{equation}
which is the equation given in \cite{jacobs2020-conjugate} as a definition of conjugate prior.
The only minor difference is that here the update map is only defined up to $(\psi\comp f)$-g.a.s.\ equality, instead of being a specified deterministic function.
We think of $\psi\colon S\to \Theta$ as a statistical model and say that it is a conjugate prior for $f$.
Its parameter space $S$ is referred to as the space of hyperparameters.
%
%The desire to obtain this more abstract perspective on \cref{jacobs-equation-per-se} was one of the main motivations of this work.
Obtaining this more abstract perspective on the definition from \cite{jacobs2020-conjugate} was one of the main motivations of this work.

It is worth briefly mentioning the further special case in which $f$ is the sampling map, although we will not make use of it.
\begin{equation}
    \label{bayes-for-X}
    B
    \left(
    \stringdiagram{
        \sdin{inT}{PX} &[1ex] \\
        \sdblackdot{cpy1} & \\
        & \sdwhitedot{f} \\
        \sdout{outT}{PX} & \sdout{outX}{X} \\
    }{
        \sdconnect{inT}{cpy1 in1}
        \sdconnect{cpy1}{f in1}
        \sdconnect{f out1}{outX in1}
        \sdconnect{cpy1 out1}{outT in1}
    }        
    \right)
=
    \stringdiagram{
        \sdin{inT}{PPX} &[1ex] \\
        \sdblackdot{cpy1} & \\
        & \sdwhitedot{f} \\[1ex]
        & \sdwhitedot{samp1} \\
        & \sdblackdot{cpy2} & \\
        \sdstoch[yshift=1mm]{bf}{\opn{Bayes}_X}{2}{2} \\
        \sdout{outT}{PPX} & \sdout{outX}{X} \\
    }{
        \sdconnect{inT}{cpy1 in1}
        \sdconnect{cpy1}{f in1}
        \sdconnect[PX]{f out1}{samp1 in1}
        \sdconnect{samp1 out1}{outX in1}
        \sdconnect{cpy2}{bf in2}
        \sdconnect{cpy1 out1}{bf in1}
        \sdconnect{bf out1}{outT in1}
    }.
\end{equation}
Here $\opn{Bayes}_X$ also performs Bayesian updating, corresponding to inference about an unknown distribution.
It takes a distribution over distributions over $X$, representing a prior, along with a sample from the unknown distribution.
Its output is the Bayesian posterior over distributions, conditioned on the sample.

We note that all the generators in this section obey the property of \emph{exchangeability}, specifically the version of that concept defined in \cite{jacobs2020-de-finetti} in the context of de Finetti's theorem.
That is, they are all machines $(S,\alpha)$  such that
    \begin{equation}
        \stringdiagram{
            \sdin[-1mm]{inS}{S} && \\
            \sdstoch{a1}{\alpha}{2}{2} && \sdinvis{ia1} \\
            \sdstoch{a2}{\alpha}{2}{2} && \sdinvis{ia2} \\
            & \sdinvis{iO11} & \sdinvis{iO12} \\
            & \sdinvis{iO21} & \sdinvis{iO22} \\
            \sdout[-1mm]{outS}{S} & \sdout{outO1}{O} & \sdout{outO2}{O} \\
        }{
            % bottom wire
            \sdconnect{inS}{a1 in1}
            \sdconnect{a1 out1}{a2 in1}
            \sdconnect{a2 out1}{outS}
            % middle output
            \sdconnect{a1 out2}{ia2 in1}
            \sdconnect{ia2 out1}{iO12 in1}
            \sdconnect{iO12 out1}{iO21 in1}
            \sdconnect{iO21 out1}{outO1 in1}
            % top output
            \sdconnect{a2 out2}{iO22 in1}
            \sdconnect{iO22 out1}{outO2 in1}
        }
        =
        \stringdiagram{
            \sdin[-1mm]{inS}{S} && \\
            \sdstoch{a1}{\alpha}{2}{2} && \sdinvis{ia1} \\
            \sdstoch{a2}{\alpha}{2}{2} && \sdinvis{ia2} \\
            \sdout[-1mm]{outS}{S} & \sdout{outO1}{O} & \sdout{outO2}{O} \\
        }{
            % bottom wire
            \sdconnect{inS}{a1 in1}
            \sdconnect{a1 out1}{a2 in1}
            \sdconnect{a2 out1}{outS}
            % top output
            \sdconnect{a1 out2}{ia2 in1}
            \sdconnect{ia2 out1}{outO2 in1}
            % middle output
            \sdconnect{a2 out2}{outO1 in1}
        }.
    \end{equation}
In our context, one of the results of \cite{jacobs2020-de-finetti} is that in \oc{Stoch} (and hence also in \oc{BorelStoch}) the category \oc{Generator}(\{0,1\}) has a terminal object, given by 
$
\stringdiagram{
        \sdin{inT}{\scriptstyle{P2}} & \\
        \sdblackdot{cpy1} & \\
        & \sdwhitedot{f} \\
        \sdout{outT}{\scriptstyle{P2}} & \sdout{outX}{\scriptstyle 2} \\
    }{
        \sdconnect{inT}{cpy1 in1}
        \sdconnect{cpy1}{f in1}
        \sdconnect{f out1}{outX in1}
        \sdconnect{cpy1 out1}{outT in1}
    }
$,
which is part of their category-theoretic treatment of de Finetti's theorem.
%$P\{0,1\}$.
%
(A much more general version of de Finetti's theorem is proved for \oc{BorelStoch} in \cite{fritz2021-de-finetti}, though in a less machine-like context.)

In the context of the machines in \cref{bayes-for-f} and \cref{bayes-for-X}, exchangeability amounts to the idea that a Bayesian reasoner should reach the same posterior from the same data, regardless of the order in which the data are presented.
(Except that here this is subject to the usual generalised-almost-surely condition.)

There is much more that can be said about exchangeability and its relationship to Bayesian inference within the framework of unifilar machines, but we will leave the topic here and return to the more general case of non-exchangeable machines in the next section.

\section{Terminal objects as ``objects of behaviours''}
\label{transducer-section}

If $\oc{UnifilarComb}(I,O)$ %or $\oc{UnifilarMealy}(I,O)$
has a terminal object then it can be seen as an ``object of behaviours,'' in much the same manner as a final coalgebra.
If such a terminal object exists we call 
it an \emph{object of transducers}.
%the elements of its state space \emph{transducers} from $I$ to $O$.
%
The intuition is that if we can meaningfully talk about its elements then they can be thought of as stochastic maps from infinite sequences of inputs to infinite sequences of outputs, subject to the causality condition described above, that each output can only depend on inputs that were received prior to it.

To illustrate this idea we prove that transducer objects always exist in \oc{Dist}, and their elements indeed have the form of stochastic maps between sequences.
For this we will need the definition of a {controlled stochastic process}.
This is a classical idea, but the category theoretic definition we give is similar to definition 9.12 of \cite{diLavore2022-coinductive}.
For further generalisations with a slightly different flavour, see section 7 of \cite{fritz2020-markov}.

\begin{definition}[controlled stochastic process]
    In a Markov category \cat{C}, we define an \emph{output-first controlled stochastic process} with input space $I$ and output space $O$ 
    as a family of morphisms $p_n\colon I^{n-1}\to O^n$ for $n \ge 1$, subject to the condition that
\begin{equation}
    \label{causality-condition-string-diagram}
    \stringdiagram{
    \sdin{inIn}{I} & \sdin{inInn}{I} & \sdin[2mm]{indots}{\vdots} & \sdin{inI1}{I} \\[2em]
    && \sdstoch{p}{p_n}{4}{4} & \\[2em]
    \sdblackdot{del1} & \sdinvis{inn} && \sdinvis{i0} \\[2em]
    & \sdout{outOnn}{O} & \sdout[2mm]{outdots}{\vdots} & \sdout{outO0}{O} \\
    }{
    % inputs
    \sdconnectlabel{n}{below=1mm}{inIn out1}{p in1}
    \sdconnect{inInn out1}{p in2}
    \sdconnect[1]{inI1 out1}{p in4}
    % outputs
    \sdconnectlabel{n}{below=1mm}{p out1}{del1 in1}
    \sdconnect{p out2}{inn in1}
    \sdconnectlabel{n-1}{below=1mm}{inn out1}{outOnn}
    \sdconnect{p out4}{i0 in1}
    \sdconnect[0]{i0 out1}{outO0}
    }
    \,\,=\,\,
    \stringdiagram{
     \sdin{inInn}{I} & \sdin[2mm]{indots}{\vdots} & \sdin{inI1}{I} \\[2em]
    & \sdstoch{p}{p_{n-1}}{3}{3} & \\[2em]
     \sdout{outOnn}{O} & \sdout[2mm]{outdots}{\vdots} & \sdout{outO0}{O} \\
    }{
    % inputs
    \sdconnectlabel{n-1}{below=1mm}{inInn out1}{p in1}
    \sdconnect[1]{inI1 out1}{p in3}
    % outputs
    \sdconnectlabel{n-1}{below=1mm}{p out1}{outOnn in1}
    \sdconnect[0]{p out3}{outO0 in1}
    },
\end{equation}
where the labels on the wires represent the indexes of the inputs and outputs.
(Note that the indices for the inputs start from 1 while the indices for the outputs start from 0, so that $p_n$ has $n-1$ inputs and $n$ outputs.
We use this convention because we consider the first output to occur ``at time 0,'' before the first input.)
An \emph{input-first controlled stochastic process} is defined similarly, but with the outputs indexed starting from 1 instead of 0, so that $p_n$ has type $I^{n-1}\to O^{n-1}$.
\end{definition}
When we say ``controlled stochastic process'' without qualification, we mean an output-first controlled stochastic process.
The condition says both that the family of distributions has to be consistent with each other, and that each output can only depend on inputs that were received prior to it.

\begin{proposition}
[\oc{Dist} has transducer objects]
\label{dist-has-all-transducers-proposition}
In \oc{Dist}, the terminal object $(\omega,T)$ of $\oc{UnifilarComb}(I,O)$ exists and is as follows.
$T$ is the set of all output-first controlled stochastic processes (in \oc{Dist}).
$\omega$ is composed of the following readout and update maps:
the readout map sends a controlled stochastic process $p$ to the distribution $p_1$, which is a distribution over $O$ with no input.
Given $i\in I$, $o\in O$ and a controlled stochastic process $p$, the update map sends $(i,o,p)$ to a delta distribution concentrated on a new controlled stochastic process $p^{i,o}$ given by
\begin{equation}
    p^{i,o}_n(o_0, \dots, o_n \mid i_1, \dots, i_n) = \frac{1}{p_1(o)}p_{n+1}(o, o_0, \dots, o_n \mid i, i_1, \dots, i_n)
\end{equation}
if $p_1(o)>0$, and to some arbitrary distribution over controlled stochastic processes otherwise.
(As such it is defined up to the appropriate generalised-almost-surely condition.)
\end{proposition}
\begin{proof}
    Given a unifilar machine $(S,\alpha)$ and a state $s\in S$, one can show inductively that under any morphism of unifilar machines $(S,\alpha)\to(T,\omega)$, the state $s$ must map to the controlled stochastic process given by
    \begin{equation}
    \stringdiagram{
     \sdin{inInn}{I} & \sdin[2mm]{indots}{\vdots} & \sdin{inI1}{I} \\[2em]
    & \sdstoch{p}{p_{n}}{3}{3} & \\[2em]
     \sdout{outOnn}{O} & \sdout[2mm]{outdots}{\vdots} & \sdout{outO0}{O} \\
    }{
    \sdconnectlabel{n}{below=1mm}{inInn out1}{p in1}
    \sdconnect[1]{inI1 out1}{p in3}
    \sdconnectlabel{n}{below=1mm}{p out1}{outOnn in1}
    \sdconnect[0]{p out3}{outO0 in1}
    }
    \,\,=\,\,
    \stringdiagram{
        & \sdin{inIn}{I} & \sdin[2mm]{indots}{\vdots} & \sdin{inI1}{I} \\
        \sddet{s}{s}{1}{1} & \sdinvis{is0} && \sdinvis{is} \\
        \sdhalfdet[yshift=1mm]{a1}{\alpha}{2}{2} & \sdinvis{ia1} && \\
        \sdblank{hdots}{\dots}{1}{1} & \sdinvis{idots} && \\
        \sdhalfdet[yshift=1mm]{a2}{\alpha}{2}{2} &&& \sdinvis{ia2}  \\
        \sdstoch{a3}{\alpha^\blob}{1}{1} & \sdinvis{ia3} && \sdinvis{ia32} \\[1em]
        \sdout{outOn}{O} & \sdout{outOnn}{O} & \sdout[2mm]{outdots}{\vdots} & \sdout{outO0}{O} \\
    }{
    % bottom wire
    \sdconnect{s out1}{a1 in1}
    \sdconnect{a1 out1}{hdots in1}
    \sdconnect{hdots out1}{a2 in1}
    \sdconnect{a2 out1}{a3 in1}
    \sdconnect[n]{a3 out1}{outOn in1}
    % first alpha
    \sdconnect[1]{inI1 out1}{is in1}
    \sdconnect{is out1}{a1 in2}
    \sdconnect{a1 out2}{ia2 in1}
    \sdconnect{ia2 out1}{ia32 in1}
    \sdconnect[0]{ia32 out1}{outO0}
    % second alpha
    \sdconnect[n]{inIn out1}{is0 in1}
    \sdconnect{is0 out1}{idots in1}
    \sdconnect{idots out1}{a2 in2}
    \sdconnect{a2 out2}{ia3 in1}
    \sdconnect[n-1]{ia3 out1}{outOnn in1}
    }\,\,.
    \end{equation}
    We give the details in \cref{dist-has-all-transducers}
\end{proof}

The update map of $(T,\omega)$ performs Bayesian conditioning:
it returns a new map from input sequences to output sequences, formed by fixing the first input and conditioning on the first output.

As usual a similar result holds for Mealy machines: $\oc{UnifilarMealy}(I,O)$ has a terminal object in \oc{Dist}, whose objects can be seen as input-first controlled stochastic processes.
The proof is similar and doesn't involve the comb condition.

One advantage of formulating transducers internally in this way is that we can consider probability distributions over them.
In particular, since $(T,\omega)$ is a terminal object it is equipped with an algebra of the monad $F\comp B$ arising from the adjunction in \cref{adjunction-result}.
This means that we can think of the unique map $B(F((T,\omega)))\to (T,\omega)$ as taking a probability distribution over transducers and returning a new transducer that represents its `average' or `expected' behaviour.
This will work in any suitable Markov category, whenever a terminal object of $\oc{UnifilarComb}(I,O)$ exists.

On the other hand, in \oc{BorelStoch} the category $\oc{UnifilarMealy}(\R,\{0,1\})$ does not have a terminal object.
Consider those machines with trivial state spaces, whose output depends only on the current input.
Specifying the behaviour of such a machine amounts to specifying a measurable map $\R\to [0,1]$.
But there is no measurable space of such functions, so there is no measurable space that includes the behaviours of all such machines.
However, we conjecture that \oc{BorelStoch} has terminal objects for $\oc{UnifilarComb}(I,O)$ and $\oc{UnifilarMealy}(I,O)$ when $I$ is a countable or finite set.

\section*{Acknowledgements}

The author thanks Martin Biehl, Matteo Capucci and the anonymous reviewers for insightful comments on the manuscript, and Martin Biehl, Simon McGregor, Timorl, Matteo Capucci and Toby Smithe for discussions that stimulated the work.
This paper was made possible through the support of Grant 62229 from the John Templeton Foundation. The opinions expressed in this publication are those of the author(s) and do not necessarily reflect the views of the John Templeton Foundation.

% \nocite{*}
\bibliographystyle{eptcs}
\nocite{fritz2023-representable-published}
\bibliography{refs}

\appendix

\section{Proof Details}

\subsection{\texorpdfstring{\oc{Dist}}{Dist} is strongly representable}
\label{dist-is-strongly-representable}

For a morphism $f\colon A\to B$ in \oc{Dist} we write $f(b\mid a)$ for the probability assigned to $b$ by the morphism $f$ when given $a$ as an input.
We have that for a given $a$ there only finitely many elements of $b$ for which $f(b\mid a)>0$, and we have $\sum_{b\in B} f(b\mid a) = 1$.

For a set $A$ the distribution object $PA$ is the set of all finitely supported probability distributions over $A$.
In other words it is the set of functions $A\to [0,1]$ that satisfy the properties above, i.e. functions that have a finite support and sum to 1.
The sampling map $\samp_A\colon PA\to A$ is given by $\samp_A(a\mid p) = p(a)$.

We now consider an arbitrary morphism $f\colon A\to X\otimes Y$ and a morphism $f^\diamond\colon A\to X\otimes PY$ such that $f^\diamond$ is deterministic given $X$ and such that \cref{strongly-representable-condition} holds.
We can factor $f^\diamond(x,p\mid a)$ as
\begin{equation}
    \stringdiagram{
        &\sdin{inA}{A}& \\
        &\sdhalfdet{f}{f^\diamond}{1}{2} & \\[1em]
        \sdout{outY}{PY} && \sdout{outX}{X} \\
    }{
        \sdconnect{inA out1}{f in1}
        \sdconnect{f out1}{outY in1}
        \sdconnect{f out2}{outX in1}
    }
    =
    \stringdiagram{
        \sdin{inA}{A} & \\
        \sdblackdot{cpy1} & \\
        & \sdstoch{h}{f^\blob}{1}{1} \\
        & \sdblackdot{cpy2} \\
        \sdstoch[yshift=1mm]{c}{c}{2}{2} \\
        \sdout{outZ}{PY} & \sdout{outX}{X}\\
    }{
        \sdconnect{inA out1}{c in1}
        \sdconnect{cpy1}{h in1}
        \sdconnect{h out1}{outX in1}
        \sdconnect{cpy2}{c in2}
        \sdconnect{c out1}{outZ in1}
    },
\end{equation}
or $f^\blob(x\mid a) c(p\mid x, a)$, for some conditional $c$.

For $f^\diamond$ to be deterministic given $X$ means that the conditional $c$ must be $f^\blob$-generalised-almost-surely deterministic (\cref{deterministic-given-x-condition}), which in \oc{Dist} means that $c(\blank\mid x, a)$ is a delta distribution whenever $f^\blob(x\mid a)>0$.
For this to be true we must have that that whenever $a$ and $x$ are such that $f^\blob(x\mid a)>0$ there exists a distribution $p_{x,y}\in PY$ such that
\begin{equation}
    \label{f-diamond-det-given-X-condition}
    f^\diamond(x,p\mid a) = \begin{cases}
        f^\blob(x\mid a) & \text{if $p=p_{a,x}$} \\
        0 & \text{otherwise.}
    \end{cases}
\end{equation}
In \oc{Dist}, \cref{strongly-representable-condition} (the definition of $f^\diamond$) amounts to
\begin{aligneq}
    \label{pxa-constraint}
    f(x,y\mid a) &= \sum_{p\in PY} f^\diamond(x,p \mid a) \samp_Y(y\mid p) \\
    &= \sum_{p\in PY} f^\diamond(x,p \mid a) p(y) \\
    &= f^\blob(x\mid a) p_{x,a}(y).    
\end{aligneq}

To show that $\oc{Dist}$ is strongly representable we have to show that $p_{x,a}$ is uniquely defined whenever $f^\blob(x\mid a)>0$.
But this follows immediately because we have, from \cref{pxa-constraint},
\begin{equation}
    \label{last-step-dist-is-strong-rep}
    p_{x,a}(y) = \frac{f(x,y\mid a)}{f^\blob(x\mid a)},
\end{equation}
which completes the proof.

Explicitly, the only choices for $f^\diamond$ are those of the form
\begin{equation}
    f^\diamond(x,p\mid a) = \begin{cases}
        \left.
            \begin{cases}
                f^\blob(x\mid a) &\text{if $p = \frac{f(x,\blank\mid a)}{f^\blob(x\mid a)}$} \\
                0 & \text{otherwise}
            \end{cases}
        \right\} & \text{if $f^\blob(x\mid a)>0$,} \\[1.5em]
        \text{arbitrary} & \text{otherwise.}
    \end{cases}
\end{equation}
The arbitrary values are subject to the constraint that $\sum_{x,p}f(x,p\mid a)=1$, as always.
They occur only outside the support of $f^\blob(\blank\mid a)$, which in \oc{Dist} means that all possible choices for $f^\diamond$ are $f^\blob$-g.a.s.\ equal, as required.

We note that the last step in the proof, \cref{last-step-dist-is-strong-rep}, would not be valid in example 3.23 of \cite{fritz2020-representable}, in which the probabilities are not real-valued.

\subsection{Comb machines morphisms commute with readout maps}
\label{morphisms-of-comb-machines}

We want to show that if $f\colon (S,\alpha)\to(T,\beta)$ is a morphism of comb machines, then $\alpha^\blob = f\comp\beta^\blob$.
By marginalising the definition of a morphism of comb machines and then substituing the definition of $\alpha^\blob$ and $\beta^\blob$ we have
    \begin{aligneq}
        \stringdiagram{
            \sdin{inPS}{S} & \sdin{inI}{I} \\
            \sdstoch[yshift=1mm]{a}{\alpha}{2}{2} & \\
            \sdstoch{f}{f}{1}{1} & \sdinvis{i2} \\
            \sdblackdot{outT} \\
             & \sdout{outO}{O} \\
        } {
            \sdconnect{inPS out1}{a in1}
            \sdconnect{a out1}{f in1}
            \sdconnect{f out1}{outT in1}
            \sdconnect{inI out1}{a in2}
            \sdconnect{a out2}{i2 in1}
            \sdconnect{i2 out1}{outO in1}
        }
        &=
        \stringdiagram{
            \sdin{inPS}{S} & \sdin{inI}{I} \\
            \sdstoch{f}{f}{1}{1} & \sdinvis{i2} \\
            \sdstoch[yshift=1mm]{a}{\beta}{2}{2} & \\
            \sdblackdot{outT} \\
             & \sdout{outO}{O} \\
        } {
            \sdconnect{inPS out1}{f in1}
            \sdconnect{f out1}{a in1}
            \sdconnect{a out1}{outT in1}
            \sdconnect{inI out1}{i2 in1}
            \sdconnect{i2 out1}{a in2}
            \sdconnect{a out2}{outO in1}
        }
        \\[1em]
        \stringdiagram{
            \sdin{inPS}{S} & \sdin{inI}{I} \\
            & \sdblackdot{del1} \\
            \sdstoch{a}{\alpha^\blob}{1}{1} & \\[1ex]
            %\sdstoch{f}{f}{1}{1} 
  %          \sdblackdot{outT} \\
             & \sdout{outO}{O} \\
        } {
            \sdconnect{inPS out1}{a in1}
            \sdconnect{inI out1}{del1 in1}
%            \sdconnect{a out1}{f in1}
 %           \sdconnect{f out1}{outT in1}
            \sdconnect{a out1}{outO in1}
        }
        &=
        \stringdiagram{
            \sdin{inPS}{S} & \sdin{inI}{I} \\
            \sdstoch{f}{f}{1}{1} & \sdblackdot{i2} \\
            \sdstoch{a}{\beta^\blob}{1}{1} & \\[1ex]
             & \sdout{outO}{O} \\
        } {
            \sdconnect{inPS out1}{f in1}
            \sdconnect{f out1}{a in1}
            \sdconnect{a out1}{outO in1}
            \sdconnect{inI out1}{i2 in1}
        } 
    \end{aligneq}
and the result follows.

\subsection{Filtering on sequences}
\label{filtering-on-sequences}

Our claim is that given a dynamical model in the form of a comb machine $(H,\kappa)$, an initial distribution over $H$, a sequence of inputs (an element of $I^n$), and an observed sequence of outputs (an element of $O^n$), the Bayesian filter (i.e.\ an update map for $B((H,\kappa))$) allows us to recover the posterior distribution over $H$, given the observations.

To make this formal let us define
\begin{equation}
\stringdiagram{
            \sdin[-1mm]{inS}{PH} & \sdin{inI}{I^n} \\
            \sdstoch{u}{\kappa^n}{2}{2} & \\
            \sdout[-1mm]{outS}{H} & \sdout{outO}{O^n} \\
        }{
            \sdconnect{inS out1}{u in1}
            \sdconnect{inI out1}{u in2}
            \sdconnect{u out1}{outS in1}
            \sdconnect{u out2}{outO in1}
        }
    \coloneqq
    \stringdiagram{
        \sdin{inP}{PH} & \sdin{inIn}{I} & \sdin{inIn1}{I} & \sdin[4mm]{indots}{\vdots} & \sdin{inI1}{I} \\
        \sdwhitedot{s} && \sdinvis{is0} && \sdinvis{is} \\[1ex]
        \sdstoch[yshift=1mm]{k1}{\kappa}{2}{2} && \sdinvis{ik1} && \\[1mm]
        \sdblank{hdots}{\dots}{1}{1} && \sdinvis{idots} && \sdinvis{idots2} \\
        \sdstoch[yshift=1mm]{k2}{\kappa}{2}{2} & \sdinvis{ik2} &&&   \\[1em]
        \sdstoch[yshift=1mm]{k3}{\kappa}{2}{2} && \sdinvis{ik3} && \sdinvis{ik32} \\[1em]
        \sdout{outH}{H} & \sdout{outOn}{O} & \sdout{outOn1}{O} & \sdout[4mm]{outdots}{\vdots} & \sdout{outO1}{O} \\
    }{
    % bottom wire
    \sdconnect{inP out1}{s in1}
    \sdconnect{s out1}{k1 in1}
    \sdconnect{k1 out1}{hdots in1}
    \sdconnect{hdots out1}{k2 in1}
    \sdconnect{k2 out1}{k3 in1}
    \sdconnect{k3 out1}{outH in1}
    % first kappa
    \sdconnect{inI1 out1}{is in1}
    \sdconnect{is out1}{k1 in2}
    \sdconnect{k1 out2}{idots2 in1}
    \sdconnect{idots2 out1}{ik32 in1}
    \sdconnect{ik32 out1}{outO1}
    % second kappa
    \sdconnect{inIn1 out1}{idots in1}
    \sdconnect{idots out1}{k2 in2}
    \sdconnect{k2 out2}{ik3 in1}
    \sdconnect{ik3 out1}{outOn1 in1}
    % third kappa
    \sdconnect{inIn out1}{ik2 in1}
    \sdconnect{ik2 out1}{k3 in2}
    \sdconnect{k3 out2}{outOn in1}
    }.
\end{equation}
What we seek is a conditional of $\kappa^n$, that is $c\colon O^n\otimes I^n\otimes PH \to H$ such that
\begin{equation}
    \stringdiagram{
        \sdin[-1mm]{inS}{PH} & \sdin{inI}{I^n} \\
            \sdstoch{u}{\kappa^n}{2}{2} & \\
            \sdout[-1mm]{outS}{H} & \sdout{outO}{O^n} \\
        }{
            \sdconnect{inS out1}{u in1}
            \sdconnect{inI out1}{u in2}
            \sdconnect{u out1}{outS in1}
            \sdconnect{u out2}{outO in1}
        }
        =
    \stringdiagram{
        \sdin{inH}{PH} & \sdin{inI}{I^n}   \\
        \sdblackdot{cpyH} & \sdblackdot{cpyI}  \\[1em]
        && \sdstoch{k}{\kappa^n}{2}{2} \\
        && \sdinvisn{i1}{2} \\
        && \sdinvisn{i2}{2} \\[1em]
        \sdstoch[yshift=1mm]{c}{c}{3}{2} \\
        \sdout{outH}{H} && \sdout[1mm]{outO}{O^n} \\
    }{
        \sdconnect{inH out1}{cpyH in1}
        \sdconnect{inI out1}{cpyI in1}
        \sdconnect{cpyH}{k in1}
        \sdconnect{cpyI}{k in2}
        \sdconnect{cpyH out1}{c in1}
        \sdconnect{cpyI out1}{c in2}
        \sdblackdotat{i1 out1}
        \sdconnect{k out1}{i1 out1}
        \sdconnect{k out2}{outO in1}
        \sdblackdotat[cpyO]{i2 out2}
        \sdconnect{cpyO}{c in3}
        \sdconnect{c out1}{outH}
    }.
\end{equation}
Then $c$, or rather $c^\squ\colon O^n\times I^n\times PH\to PH$, is the map that takes the observed output sequence, the input sequence and returns the prior over $H$ to the posterior over $H$.
Our claim is that such a conditional $c$ is given by composing $n$ instances of $u$ as follows,
\begin{equation}
    \stringdiagram{
        \sdin{inH}{PH} & \sdin{inI}{I^n} & \sdin{inO}{O^n} \\[1ex]
        \sdstoch[yshift=1mm]{c}{c}{3}{2} \\
        \sdout{outH}{H} \\
    }{
        \sdconnect{inH out1}{c in1}
        \sdconnect{inI out1}{c in2}
        \sdconnect{inO out1}{c in3}
        \sdconnect{c out1}{outH in1}
    }
    \coloneqq
    \stringdiagram{
        \sdin{inH}{PH} & \sdin{inI1}{I} & \sdin[4mm]{inIdots}{\vdots} & \sdin{inI2}{I} & \sdin{inO1}{O} & \sdin{inOdots}{\vdots} & \sdin{inO2}{O} &  \\[2em]
        \sdstoch[yshift=1mm]{u1}{u}{3}{2} \\
        \sdblank{hdots}{\dots}{1}{1} & \sdinvis{iI} &&& \sdinvis{iO} \\[1em]
        \sdstoch[yshift=1mm]{u2}{u}{3}{2} \\
        \sdwhitedot{samp} \\
        \sdout{outH}{H} \\
    } {
        \sdconnect{inH out1}{u1 in1}
        \sdconnect{inI2 out1}{u1 in2}
        \sdconnect{inO2 out1}{u1 in3}
        \sdconnect{u1 out1}{hdots in1}
        \sdconnect{inI1 out1}{iI in1}
        \sdconnect{inO1 out1}{iO in1}
        \sdconnect{iI out1}{u2 in2}
        \sdconnect{iO out1}{u2 in3}
        \sdconnect{hdots out1}{u2 in1}
        \sdconnect{u2 out1}{outH in1}
    }
\end{equation}
where $u$ is an update map for $B((H,\kappa))$.

The proof of this is by induction and can be expressed as a lengthy but straightforward string diagram calculation; we omit it for reasons of space.
A similar result holds in the Mealy machine case.

A proof of a similar statement was given previously by the author and colleagues in \cite{virgo2021-interpreting} (proposition 2 in appendix B.2), although with somewhat different definitions since that paper does not consider representable Markov categories.

It is worth remarking that the omitted proof uses the fact that $u$ is $(P\kappa^\blob\comp\opn{samp})$-g.a.s.\ deterministic, since this fact is not directly used in the proofs of our other results.
(The assumption is nevertheless needed, in order to apply the defining property of a strongly representable Markov category.)

\subsection{\texorpdfstring{$\opn{Bayes}_f$}{Bayes-f} does Bayes}
\label{bayes-f-does-bayes}

We want to show that the morphism $\opn{Bayes}_f$ in the expression
\begin{equation}
%    \label{bayes-for-f}
    B\left(
    \stringdiagram{
        \sdin{inT}{\Theta} & \\
        \sdblackdot{cpy1} & \\
        & \sdstoch{f}{f}{1}{1} \\
        \sdout{outT}{\Theta} & \sdout{outX}{X} \\
    }{
        \sdconnect{inT}{cpy1 in1}
        \sdconnect{cpy1}{f in1}
        \sdconnect{f out1}{outX in1}
        \sdconnect{cpy1 out1}{outT in1}
    }
    \right)
    =
    \stringdiagram{
        \sdin{inT}{P\Theta} & \\
        \sdblackdot{cpy1} & \\
        & \sddet{f}{Pf}{1}{1} \\
        & \sdwhitedot{samp1} \\
        & \sdblackdot{cpy2} & \\
        \sdstoch[yshift=1.1mm]{bf}{\opn{Bayes}_f}{2}{2} \\
        \sdout{outT}{P\Theta} & \sdout{outX}{X} \\
    }{
        \sdconnect{inT}{cpy1 in1}
        \sdconnect{cpy1}{f in1}
        \sdconnect{f out1}{samp1 in1}
        \sdconnect{samp1 out1}{outX in1}
        \sdconnect{cpy2}{bf in2}
        \sdconnect{cpy1 out1}{bf in1}
        \sdconnect{bf out1}{outT in1}
    }
\end{equation}
can be seen as performing a Bayesian update.

For this we recall the definition of a \emph{Bayesian inverse} from \cite{cho2017-disintegration}, which we generalise slightly by allowing the prior to depend on a parameter.
In a Markov category $\cat{C}$, given a morphism $f\colon\Theta\to X$ and a morphism $p\colon \Phi\to \Theta$ called the prior, we say that $f^\dagger_p\colon X\to \Theta$ is a \emph{Bayesian inverse of $f\colon \Theta\to X$ with respect to $p$} if
\begin{equation}
    \label{bayesian-inverse-definition}
    \stringdiagram{
        \sdin{inT}{\Phi} &\\
        \sdstoch{p}{p}{1}{1} &\\
        \sdblackdot{cpy} &\\
        & \sdstoch{f}{f}{1}{1} &\\
        \sdout{outT}{\Theta} & \sdout{outX}{X} \\
    }{
        \sdconnect{inT out1}{p in1}
        \sdconnect{p out1}{outT in1}
        \sdconnect{cpy}{f in1}
        \sdconnect{f out1}{outX in1}
    }
    =
    \stringdiagram{
        \sdin{inT}{\Phi} & \\
        \sdblackdot{cpy1} & \\
        & \sdstoch{p}{p}{1}{1} \\
        & \sdstoch{f}{f}{1}{1} \\
        & \sdblackdot{cpy2} & \\
        \sdstoch[yshift=1.3mm]{fd}{f^\dagger_p}{2}{2} & \\
        \sdout{outT}{\Theta} & \sdout{outX}{X} \\
    }{
        \sdconnect{inT}{cpy1 in1}
        \sdconnect{cpy1}{p in1}
        \sdconnect{p out1}{f in1}
        \sdconnect{f out1}{outX in1}
        \sdconnect{cpy2}{fd in2}
        \sdconnect{cpy1 out1}{fd in1}
        \sdconnect{fd out1}{outT in1}
    }.
\end{equation}
If \cat{C} has conditionals then such Bayesian inverse exists for any $f$ and $p$.
It is not necessarily unique, but like any conditional it is unique up to $(p\comp f)$-g.a.s.\ equivalence.

If \cat{C} is representable we can take $\Phi = P\Theta$ and $p = \samp_\Theta$, yielding
\begin{aligneq}
    \label{universal-bayesian-inverse-definition}
    \stringdiagram{
        \sdin{inT}{P\Theta} &\\
        \sdwhitedot{p} &\\
        \sdblackdot{cpy} &\\
        & \sdstoch{f}{f}{1}{1} &\\
        \sdout{outT}{\Theta} & \sdout{outX}{X} \\
    }{
        \sdconnect{inT out1}{p in1}
        \sdconnect{p out1}{outT in1}
        \sdconnect{cpy}{f in1}
        \sdconnect{f out1}{outX in1}
    }
    &=
    \stringdiagram{
        \sdin{inT}{P\Theta} & \\
        \sdblackdot{cpy1} & \\
        & \sdwhitedot{p} \\
        & \sdstoch{f}{f}{1}{1} \\
        & \sdblackdot{cpy2} & \\
        \sdstoch[yshift=1.35mm]{fd}{{f^\dagger_{\samp}}}{2}{2} & \\
        \sdout{outT}{\Theta} & \sdout{outX}{X} \\
    }{
        \sdconnect{inT}{cpy1 in1}
        \sdconnect{cpy1}{p in1}
        \sdconnect{p out1}{f in1}
        \sdconnect{f out1}{outX in1}
        \sdconnect{cpy2}{fd in2}
        \sdconnect{cpy1 out1}{fd in1}
        \sdconnect{fd out1}{outT in1}
    } \\[1em]
    &=
    \stringdiagram{
        \sdin{inT}{P\Theta} & \\
        \sdblackdot{cpy1} & \\
        & \sddet{f}{Pf}{1}{1} \\
        & \sdwhitedot{s} \\
        & \sdblackdot{cpy2} & \\
        \sdstoch[yshift=1.35mm]{fd}{{f^\dagger_{\samp}}}{2}{2} & \\
        \sdout{outT}{\Theta} & \sdout{outX}{X} \\
    }{
        \sdconnect{inT}{cpy1 in1}
        \sdconnect{cpy1}{f in1}
        \sdconnect{f out1}{s in1}
        \sdconnect{s out1}{outX in1}
        \sdconnect{cpy2}{fd in2}
        \sdconnect{cpy1 out1}{fd in1}
        \sdconnect{fd out1}{outT in1}
    }.
\end{aligneq}
The morphism $f^\dagger_{\samp}$ can be thought of as performing a Bayesian inversion of $f$ with respect to \emph{any} prior, since it takes an element of $P\Theta$ as an input.

If $\cat{C}$ is strongly representable, then from the definition of $\opn{Bayes}_f$ we have
\begin{equation}
    \stringdiagram{
        \sdin{inT}{P\Theta} &\\
        \sdwhitedot{p} &\\
        \sdblackdot{cpy} &\\
        & \sdstoch{f}{f}{1}{1} &\\
        \sdout{outT}{\Theta} & \sdout{outX}{X} \\
    }{
        \sdconnect{inT out1}{p in1}
        \sdconnect{p out1}{outT in1}
        \sdconnect{cpy}{f in1}
        \sdconnect{f out1}{outX in1}
    }
    =
    \stringdiagram{
        \sdin{inT}{P\Theta} & \\
        \sdblackdot{cpy1} & \\
        & \sddet{f}{Pf}{1}{1} \\
        & \sdwhitedot{samp1} \\
        & \sdblackdot{cpy2} & \\
        \sdstoch[yshift=1.1mm]{bf}{\opn{Bayes}_f}{2}{2} \\
        \sdwhitedot{samp} \\
        \sdout{outT}{\Theta} & \sdout{outX}{X} \\
    }{
        \sdconnect{inT}{cpy1 in1}
        \sdconnect{cpy1}{f in1}
        \sdconnect{f out1}{samp1 in1}
        \sdconnect{samp1 out1}{outX in1}
        \sdconnect{cpy2}{bf in2}
        \sdconnect{cpy1 out1}{bf in1}
        \sdconnect{bf out1}{outT in1}
    }.
\end{equation}
We conclude that up to $Pf\comp\samp$-g.a.s.\ equivalence, $\opn{Bayes}_f$ is the same as $(f^\dagger_{\opn{samp}})^\squ$, the deterministic version of $f^\dagger_{\samp}$.
It takes as input a prior in $P\Theta$ and a data point from $X$, and gives as output the Bayesian posterior as an element of $P\Theta$.

\subsection{\texorpdfstring{\oc{Dist}}{Dist} has transducers}
\label{dist-has-all-transducers}

We write $\omega^\blob\colon T\to O$ for the readout map of the unifilar machine $(T,\omega)$ defined in \cref{dist-has-all-transducers-proposition}.
We want to show that this unifilar machine is the terminal object of $\oc{UnifilarComb}(I,O)$.

We note that the existence of the terminal object could be proven in a different way, by noting that comb machines in Dist can be formulated as coalgebras of a polynomial functor on $\oc{Set}$.
The existence of a terminal object (a.k.a.\ final coalgebra) is then a standard result.
However, the proof below leads more directly to the probabilistic interpretation in terms of controlled stochastic processes.

As mentioned in the main text we will show  that given a unifilar machine $(S,\alpha)$ and a state $s\in S$, under any morphism of unifilar machines $(S,\alpha)\to(T,\omega)$, the state $s$ must map to the controlled stochastic process given by
    \begin{equation}
    \stringdiagram{
     \sdin{inInn}{I} & \sdin[2mm]{indots}{\vdots} & \sdin{inI1}{I} \\[2em]
    & \sdstoch{p}{p_{n}}{3}{3} & \\[2em]
     \sdout{outOnn}{O} & \sdout[2mm]{outdots}{\vdots} & \sdout{outO0}{O} \\
    }{
    \sdconnectlabel{n}{below=1mm}{inInn out1}{p in1}
    \sdconnect[1]{inI1 out1}{p in3}
    \sdconnectlabel{n}{below=1mm}{p out1}{outOnn in1}
    \sdconnect[0]{p out3}{outO0 in1}
    }
    \,\,=\,\,
    \stringdiagram{
        & \sdin{inIn}{I} & \sdin[2mm]{indots}{\vdots} & \sdin{inI1}{I} \\
        \sddet{s}{s}{1}{1} & \sdinvis{is0} && \sdinvis{is} \\
        \sdhalfdet[yshift=1mm]{a1}{\alpha}{2}{2} & \sdinvis{ia1} && \\
        \sdblank{hdots}{\dots}{1}{1} & \sdinvis{idots} && \\
        \sdhalfdet[yshift=1mm]{a2}{\alpha}{2}{2} &&& \sdinvis{ia2}  \\
        \sdstoch{a3}{\alpha^\blob}{1}{1} & \sdinvis{ia3} && \sdinvis{ia32} \\[1em]
        \sdout{outOn}{O} & \sdout{outOnn}{O} & \sdout[2mm]{outdots}{\vdots} & \sdout{outO0}{O} \\
    }{
    % bottom wire
    \sdconnect{s out1}{a1 in1}
    \sdconnect{a1 out1}{hdots in1}
    \sdconnect{hdots out1}{a2 in1}
    \sdconnect{a2 out1}{a3 in1}
    \sdconnect[n]{a3 out1}{outOn in1}
    % first alpha
    \sdconnect[1]{inI1 out1}{is in1}
    \sdconnect{is out1}{a1 in2}
    \sdconnect{a1 out2}{ia2 in1}
    \sdconnect{ia2 out1}{ia32 in1}
    \sdconnect[0]{ia32 out1}{outO0}
    % second alpha
    \sdconnect[n]{inIn out1}{is0 in1}
    \sdconnect{is0 out1}{idots in1}
    \sdconnect{idots out1}{a2 in2}
    \sdconnect{a2 out2}{ia3 in1}
    \sdconnect[n-1]{ia3 out1}{outOnn in1}
    }\,\,.
    \end{equation}
%
%We will write $(T,\omega)$ for the unifilar machine thus defined, with 
%
%We write $\omega^\blob\colon T\to O$ for the readout map of the unifilar machine $(T,\omega)$ defined in \cref{dist-has-all-transducers-proposition}.
%
%We want to show that this unifilar machine is the terminal one.
%and $u_\omega$ for the update map.
%
%
It is straightforward to show inductively that for $p\in T$ we have
\begin{equation}
    \label{transducer-identity}
    \stringdiagram{
     \sdin{inInn}{I} & \sdin[2mm]{indots}{\vdots} & \sdin{inI1}{I} \\[2em]
    & \sdstoch{p}{p_{n}}{3}{3} & \\[2em]
     \sdout{outOnn}{O} & \sdout[2mm]{outdots}{\vdots} & \sdout{outO0}{O} \\
    }{
    \sdconnectlabel{n}{below=1mm}{inInn out1}{p in1}
    \sdconnect[1]{inI1 out1}{p in3}
    \sdconnectlabel{n}{below=1mm}{p out1}{outOnn in1}
    \sdconnect[0]{p out3}{outO0 in1}
    }
    \,\,=\,\,
    \stringdiagram{
        & \sdin{inIn}{I} & \sdin[2mm]{indots}{\vdots} & \sdin{inI1}{I} \\
        \sddet{s}{p}{1}{1} & \sdinvis{is0} && \sdinvis{is} \\
        \sdhalfdet[yshift=1mm]{a1}{\omega}{2}{2} & \sdinvis{ia1} && \\
        \sdblank{hdots}{\dots}{1}{1} & \sdinvis{idots} && \\
        \sdhalfdet[yshift=1mm]{a2}{\omega}{2}{2} &&& \sdinvis{ia2}  \\
        \sdstoch{a3}{\omega^\blob}{1}{1} & \sdinvis{ia3} && \sdinvis{ia32} \\[1em]
        \sdout{outOn}{O} & \sdout{outOnn}{O} & \sdout[2mm]{outdots}{\vdots} & \sdout{outO0}{O} \\
    }{
    % bottom wire
    \sdconnect{s out1}{a1 in1}
    \sdconnect{a1 out1}{hdots in1}
    \sdconnect{hdots out1}{a2 in1}
    \sdconnect{a2 out1}{a3 in1}
    \sdconnect[n]{a3 out1}{outOn in1}
    % first alpha
    \sdconnect[1]{inI1 out1}{is in1}
    \sdconnect{is out1}{a1 in2}
    \sdconnect{a1 out2}{ia2 in1}
    \sdconnect{ia2 out1}{ia32 in1}
    \sdconnect[0]{ia32 out1}{outO0}
    % second alpha
    \sdconnect[n]{inIn out1}{is0 in1}
    \sdconnect{is0 out1}{idots in1}
    \sdconnect{idots out1}{a2 in2}
    \sdconnect{a2 out2}{ia3 in1}
    \sdconnect[n-1]{ia3 out1}{outOnn in1}
    }\,\,,
\end{equation}
where $\omega^\blob$ is the readout map of $(T,\omega)$. 

Now suppose we are given a unifilar machine $(S, \alpha)$, and write $\alpha^\blob\colon S\to O$ for its readout map.
Consider a map of unifilar machines $h\colon (S,\alpha)\to (T,\omega)$.
Our goal is to show that such a map always exists and is uniquely defined.
Let $s\in S$ be state of $\alpha$ and let $p=h(s)$ be the transducer that it maps to under $h$.
We then calculate

\begin{multline}
\qquad\qquad
    \stringdiagram{
     \sdin{inInn}{I} & \sdin[2mm]{indots}{\vdots} & \sdin{inI1}{I} \\[2em]
    & \sdstoch{p}{p_{n}}{3}{3} & \\[2em]
     \sdout{outOnn}{O} & \sdout[2mm]{outdots}{\vdots} & \sdout{outO0}{O} \\
    }{
    \sdconnectlabel{n}{below=1mm}{inInn out1}{p in1}
    \sdconnect[1]{inI1 out1}{p in3}
    \sdconnectlabel{n}{below=1mm}{p out1}{outOnn in1}
    \sdconnect[0]{p out3}{outO0 in1}
    }
    \,\, = \,\,
    \stringdiagram{
        & \sdin{inIn}{I} & \sdin[2mm]{indots}{\vdots} & \sdin{inI1}{I} \\
        \sddet{s}{s}{1}{1} & \sdinvis{is0} && \sdinvis{is} \\
        \sddet{h}{h}{1}{1} & && \\
        \sdhalfdet[yshift=1mm]{a1}{\omega}{2}{2} & \sdinvis{ia1} && \\
        \sdblank{hdots}{\dots}{1}{1} & \sdinvis{idots} && \\
        \sdhalfdet[yshift=1mm]{a2}{\omega}{2}{2} &&& \sdinvis{ia2}  \\
        \sdstoch{a3}{\omega^\blob}{1}{1} & \sdinvis{ia3} && \sdinvis{ia32} \\[1em]
        \sdout{outOn}{O} & \sdout{outOnn}{O} & \sdout[2mm]{outdots}{\vdots} & \sdout{outO0}{O} \\
    }{
    % bottom wire
    \sdconnect{s out1}{a1 in1}
    \sdconnect{a1 out1}{hdots in1}
    \sdconnect{hdots out1}{a2 in1}
    \sdconnect{a2 out1}{a3 in1}
    \sdconnect[n]{a3 out1}{outOn in1}
    % first alpha
    \sdconnect[1]{inI1 out1}{is in1}
    \sdconnect{is out1}{a1 in2}
    \sdconnect{a1 out2}{ia2 in1}
    \sdconnect{ia2 out1}{ia32 in1}
    \sdconnect[0]{ia32 out1}{outO0}
    % second alpha
    \sdconnect[n]{inIn out1}{is0 in1}
    \sdconnect{is0 out1}{idots in1}
    \sdconnect{idots out1}{a2 in2}
    \sdconnect{a2 out2}{ia3 in1}
    \sdconnect[n-1]{ia3 out1}{outOnn in1}
    }
    \\[1em]
    =\,\,
    \stringdiagram{
        & \sdin{inIn}{I} & \sdin[2mm]{indots}{\vdots} & \sdin{inI1}{I} \\
        \sddet{s}{s}{1}{1} & \sdinvis{is0} && \sdinvis{is} \\
        \sdhalfdet[yshift=1mm]{a1}{\alpha}{2}{2} & \sdinvis{ia1} && \\
        \sdblank{hdots}{\dots}{1}{1} & \sdinvis{idots} && \\
        \sdhalfdet[yshift=1mm]{a2}{\alpha}{2}{2} &&& \sdinvis{ia2}  \\
        \sddet{h}{h}{1}{1} & \sdinvis{ih} &&& \\
        \sdstoch{a3}{\omega^\blob}{1}{1} & \sdinvis{ia3} && \sdinvis{ia32} \\[1em]
        \sdout{outOn}{O} & \sdout{outOnn}{O} & \sdout[2mm]{outdots}{\vdots} & \sdout{outO0}{O} \\
    }{
    % bottom wire
    \sdconnect{s out1}{a1 in1}
    \sdconnect{a1 out1}{hdots in1}
    \sdconnect{hdots out1}{a2 in1}
    \sdconnect{a2 out1}{a3 in1}
    \sdconnect[n]{a3 out1}{outOn in1}
    % first alpha
    \sdconnect[1]{inI1 out1}{is in1}
    \sdconnect{is out1}{a1 in2}
    \sdconnect{a1 out2}{ia2 in1}
    \sdconnect{ia2 out1}{ia32 in1}
    \sdconnect[0]{ia32 out1}{outO0}
    % second alpha
    \sdconnect[n]{inIn out1}{is0 in1}
    \sdconnect{is0 out1}{idots in1}
    \sdconnect{idots out1}{a2 in2}
    \sdconnect{a2 out2}{ih in1}
    \sdconnect{ih out1}{ia3 in1}
    \sdconnect[n-1]{ia3 out1}{outOnn in1}
    }
    \,\,=\,\,
    \stringdiagram{
        & \sdin{inIn}{I} & \sdin[2mm]{indots}{\vdots} & \sdin{inI1}{I} \\
        \sddet{s}{s}{1}{1} & \sdinvis{is0} && \sdinvis{is} \\
        \sdhalfdet[yshift=1mm]{a1}{\alpha}{2}{2} & \sdinvis{ia1} && \\
        \sdblank{hdots}{\dots}{1}{1} & \sdinvis{idots} && \\
        \sdhalfdet[yshift=1mm]{a2}{\alpha}{2}{2} &&& \sdinvis{ia2}  \\
        \sdstoch{a3}{\alpha^\blob}{1}{1} & \sdinvis{ia3} && \sdinvis{ia32} \\[1em]
        \sdout{outOn}{O} & \sdout{outOnn}{O} & \sdout[2mm]{outdots}{\vdots} & \sdout{outO0}{O} \\
    }{
    % bottom wire
    \sdconnect{s out1}{a1 in1}
    \sdconnect{a1 out1}{hdots in1}
    \sdconnect{hdots out1}{a2 in1}
    \sdconnect{a2 out1}{a3 in1}
    \sdconnect[n]{a3 out1}{outOn in1}
    % first alpha
    \sdconnect[1]{inI1 out1}{is in1}
    \sdconnect{is out1}{a1 in2}
    \sdconnect{a1 out2}{ia2 in1}
    \sdconnect{ia2 out1}{ia32 in1}
    \sdconnect[0]{ia32 out1}{outO0}
    % second alpha
    \sdconnect[n]{inIn out1}{is0 in1}
    \sdconnect{is0 out1}{idots in1}
    \sdconnect{idots out1}{a2 in2}
    \sdconnect{a2 out2}{ia3 in1}
    \sdconnect[n-1]{ia3 out1}{outOnn in1}
    }\,\,.
\end{multline}
The second equality is by induction, moving $h$ to the right across the chain of $n$ morphisms.

This leaves us with exactly one choice for $p=h(s)$ for each $s\in S$, and we conclude that the map $h\colon (S,\alpha)\to (T,\omega)$ is unique.

\end{document}